\documentclass{article}
\usepackage[utf8]{inputenc}
\usepackage[T2A]{fontenc}
\usepackage{amsmath,amssymb,amsthm}
\usepackage[a4paper,hmargin=2.5cm,vmargin=2.5cm]{geometry}
\usepackage[english]{babel}
\usepackage{amsmath,amssymb,amsthm,amsmath}
\usepackage{color,graphics}
\usepackage[usenames,dvipsnames]{pstricks}
\usepackage{pst-grad} 
\usepackage{pst-plot} 
\DeclareGraphicsExtensions{.pdf,.png,.jpg}
\usepackage{tikz}
\usepackage[utf8]{inputenc}
\usepackage[english]{babel}
\usepackage{pdfpages}
\usepackage{graphicx}
\usepackage{hyperref}
\usepackage[all]{xy}
\usepackage{mathtools}
\usepackage[bottom]{footmisc}
\usepackage{ifthen}
\usepackage{xfrac}
\usepackage{comment}
\usepackage{todonotes}
\usepackage{authblk}

\newtheorem{theorem}{Theorem}[section] 

\newtheorem{lemma}[theorem]{Lemma}
\newtheorem{proposition}[theorem]{Proposition}
\newtheorem{corollary}[theorem]{Corollary}

\newtheorem{problem}[theorem]{Problem}
\newtheorem{conjecture}[theorem]{Conjecture}

\theoremstyle{remark}
\newtheorem{example}[theorem]{Example}
\newtheorem{remark}[theorem]{Remark}
\theoremstyle{definition}
\newtheorem{definition}[theorem]{Definition}

\title{Matroids arising from algebraic shifting}
\author[1]{Lazar Guterman{\thanks{lazar.guterman@mail.huji.ac.il}}}
\author[1,2,3]{Eran Nevo\thanks{eran.nevo@uva.es, nevo@math.huji.ac.il \\ Both authors were partially supported by the Israel Science Foundation grant ISF-687/24.}}
\affil[1]{{\it \normalsize Einstein Institute of Mathematics, Hebrew University}}
\affil[2]{{\it \normalsize Mathematics Research Institute, Universidad de Valladolid}}
\affil[3]{{\it \normalsize CMSA, Harvard University}}

\date{}

\begin{document}

\maketitle

\begin{abstract}

 We characterize the shifted simple graphs and the $3$-uniform shifted hypergraphs whose inverse image under exterior shifting is the set of bases of a matroid: 
 those are exactly the
 hypergraphs whose hyperedges form an initial lex-segment. There are several examples of known matroids arising in this way: the simplicial matroid, the hyperconnectivity matroid and the area-rigidity matroid. For $k\ge 4$, we provide a similar characterization for shifted  $k$-uniform hypergraphs satisfying an additional combinatorial condition.

 For symmetric shifting, we prove an analogous characterization for shifted simple graphs, where the classical generic rigidity matroid is an example of a matroid arising in this way. 
\end{abstract}

\section{Introduction}
A collection $H$ of $k$-element subsets of $[n]$  is {\itshape shifted} if for every $i<j$, $j\in S\in H$, $i\notin S$, also $(S\setminus\{j\})\cup \{i\}\in H$. After fixing a field, a canonical operation of {\itshape algebraic shifting}, which associates with every simplicial complex $K$ a shifted simplicial complex $\Delta(K)$, 
was introduced by Kalai \cite{K ch, K Alg Sh}. This operation preserves some important properties of $K$, such as its $f$-vector and $\beta$-vector, while loses others, such as the homotopy type of $K$. 
Various properties of $K$ are described in terms of its shifting, such as its homology and Cohen-Macaulyness over the fixed field.

Two different variations of algebraic shifting were introduced 
by Kalai: {\itshape exterior 
shifting} and {\itshape symmetric 
shifting}. Given a simplicial complex $K$ we denote by $\Delta^e(K)$ and $\Delta^s(K)$ its exterior and symmetric shifting respectively, and sometimes omit the superscript for  statements that hold for both types of algebraic shifting. Throughout the paper all results on exterior shifting are satisfied over any field, and all results on symmetric shifting are satisfied over any field of characteristic zero.

We start our exposition with a simple known example. Let $n$ be a positive integer. The {\itshape graphic matroid} of the complete graph on $n$ vertices is a collection of all spanning trees on the vertex set $[n]$. Let $G$ be such a spanning tree. Note that since algebraic shifting preserves Betti numbers, for every edge $e\in\Delta(G)$, we have $1\in e$. Note also that since algebraic shifting preserves $f$-vector, it then follows that $\Delta(G)=\{12,13,\ldots,1n\}$. Conversely, one can similarly verify that every $G\in\Delta^{-1}(\{12,13,\ldots,1n\})$ is a spanning tree on $[n]$. Thus, we conclude that $\Delta^{-1}(\{12,13,\ldots,1n\})$ coincides with the graphic matroid of the complete graph on $n$ vertices. The main goal of this paper is to study this phenomenon in a general setting.

Let $H$ be a shifted $k$-uniform hypergraph on $[n]$ (equivalently, the collection of facets of a pure simplicial complex). 
Denote by
$$M(H)=\left\{G\subseteq \binom{[n]}k|\ \Delta^e(G)=H\right\}$$
all $k$-uniform hypergraphs whose exterior shifting is $H$. 
We say that
$H$ is {\itshape matroidal} if $M(H)$ is the set of bases of a matroid; we abuse notation and also denote that matroid by $M(H)$.

For example, if $H$ is the collection of $k$-subsets of $[n]$ containing $1$, then $M(H)$ is the Crapo-Rota simplicial matroid \cite{CR1, CR2}, 
whose independent sets are the $(k-1)$-acyclic collections of $k$-subsets of $[n]$, namely those that do not support any $(k-1)$-homology cycle over a fixed 
field; see e.g.
~\cite[Chapter 6]{White} and \cite[Section 4, Remark (2)]{K Symm} for further background. If $H=\{ij:\ 1\le i<j\le n, \ i\le k\}$ then $M(H)$ is Kalai's $k$-hyperconnectivity matroid~\cite{K Hyper}; see e.g. \cite[Section 4, Remark (2)]{K Symm} and \cite[Section 2.7]{K Alg Sh} for further background. If $H=\{1ij:\ 2\le i<j\le n, \ i\le 3\}$ then $M(H)$ is the area-rigidity matroid, as recently proved in~\cite[Theorem 1.2]{BNP}. 

We conjecture a combinatorial characterization of the algebraic property of being \emph{matroidal}:
\begin{conjecture}\label{conj:any_k}
For every natural numbers $1\le k\le n$, a shifted $k$-uniform hypergraph $H\subseteq\binom{[n]}{k}$ is matroidal if and only if the hyperedges of $H$ form an initial segment w.r.t. the lexicogrpaphic order
(initial lex-segment for short).
\end{conjecture}

The "if" part follows easily from a result of Kalai~\cite[Section 4]{K Symm}. 
We prove the "only if" part for $k=2$ (i.e. for graphs) and for $k=3$:
\begin{theorem} \label{thm: main thm 2 and 3}
For $k=2,3$, 
a shifted 
$k$-uniform 
hypergraph $H\subseteq\binom{[n]}{k}$ is matroidal if and only if the (hyper)edges of $H$ form an 
initial lex-segment.
\end{theorem}

For $k\ge 4$ we prove the "only if" part for $k$-uniform hypergraphs satisfying some technical condition on their structure. 

\begin{theorem} \label{main theorem}
Let $k\ge 4$ and
let $H\subseteq\binom{[n]}{k}$ be a shifted $k$-uniform hypergraph such that $s_H\notin H$, where $s_H$ is defined in terms of the lex-maximum element of $H$ by \eqref{S for case 1}-\eqref{S for case 2}. Then $H$ is matroidal if and only if the hyperedges of $H$ form an initial lex-segment.
\end{theorem}

For the "only if" part of Theorem \ref{thm: main thm 2 and 3} we first make a reduction of the general problem to a certain restricted subset of hypergraphs; and for each of them we find a suitable permutation $\pi:[n]\rightarrow [n]$ such that the pair $(H, \pi(H))$, where $\pi(H)=\{\pi(T):\ T\in H\}$, violates the exchange axiom among the elements of $M(H)$. For the "only if" part of Theorem \ref{main theorem} we find a suitable such permutation $\pi=\pi(s_M)$ for each $H$.

The analog of Theorem~\ref{thm: main thm 2 and 3} 
for symmetric algebraic shifting is proved 
for $k=2$ in Proposition~\ref{thm: symmetric case for graphs}. 
For $H=\{ij:\ 1\le i\le j \le n, i\le k\}$ the corresponding matroid is the (generic) $k$-rigidity matroid, see e.g. \cite{Lee} and  \cite[Section 2.7]{K Alg Sh}, \cite[Section 3.2]{Nevo phd} for further background, references and for a discussion of this correspondence.
\

This paper is organized as follows. In Section \ref{sec: preliminaries} we give a background on algebraic shifting and matroids. In Section \ref{sec: Preparatory lemmas} we prove lemmas needed for proofs of Theorem \ref{thm: main thm 2 and 3} and Theorem \ref{main theorem}, and in Section \ref{sec: Proofs of main results} we prove these theorems. In Section \ref{sec: Matroids arising from symmetric shifting} we consider matroids arising from symmetric shifting, and in Proposition \ref{thm: symmetric case for graphs} classify all matroidal  
simple graphs 
w.r.t. symmetric shifting. In Section \ref{sec: Some partial results} we replace the matroidal property by a stronger algebraic property in Conjecture~\ref{conj:any_k} that allows us to prove the resulted "iff" combinatorial characterization.  
In Section \ref{sec: Concluding remarks} we conclude with related open problems.

\section{Preliminaries} \label{sec: preliminaries}
\subsection{Hypergraphs and simplicial complexes} \label{subsection: hypergraphs}
Let $[n]:=\{1,2,\ldots,n\}$. Denote by $\binom{[n]}{k}$ the set of subsets of $[n]$ of size $k$. A {\itshape simplicial complex} $K$ is a collection of subsets of $[n]$ which is closed under inclusion, i.e. $S\subseteq T\in K$ implies $S\in K$. The elements of $K$ are called {\itshape faces}. The {\itshape $i$-th skeleton} of $K$ is $K_i=\{S\in K:|S|=i+1\}$. Faces which belong to $K_i$ are {\itshape $i$-dimensional} faces. The $0$-dimensional faces are called {\itshape vertices} and the maximum faces with respect to inclusion are called {\itshape facets}. If all facets have the same dimension, $K$ is {\itshape pure}. 

The {\itshape $f$-vector} (face vector) of $K$ is $f(K)=(f_{-1},f_0,f_1,\ldots)$ where $f_i=|K_i|$. The {\itshape dimension} of $K$ is $\dim (K):=\max\{i:f_i(K)\ne0\}$. In particular, a $1$-dimensional simplicial complex is a {\itshape simple graph}, whose $1$-dimensional faces are called {\itshape edges}. 

The join of two simplicial complexes $K$ and $L$ with disjoint sets of vertices is the simplicial complex $K*L=\{S\cup T:S\in K,T\in L\}$. The {\itshape cone} over a simplicial complex $K$ is the join of $K$ with a single vertex, denoted by $\text{Cone}(K)$.

Denote by $\beta_i(K,\mathbb{L})$ the $i$-th Betti number of $K$ with coefficients in $\mathbb L$, or simply $\beta_i(K)$ if $\mathbb{L}$ was specified earlier. Denote by $\beta(K)=(\beta_0(K),\beta_1(K),\ldots)$ the {\itshape Betti vector} of $K$.

A simplicial complex $K$ is {\itshape shifted} if for every $i<j$, $j\in S\in K$, also $(S\setminus\{j\})\cup \{i\}\in K$. Let $<_p$ be the product partial order on equal sized ordered subsets of $\mathbb N$. That is, for $S=\{s_1<\ldots<s_k\}$ and $T=\{t_1<\ldots<t_k\}$ $S\le_p T$ iff $s_j\le t_j$ for every $1\le j\le k$. Then $K$ is shifted iff $S<_pT\in K$ implies $S\in K$. Every simplicial complex has a shifted simplicial complex with the same $f$-vector given by the {\itshape combinatorial shifting} $K\mapsto\Delta^c(K)$ introduced in \cite{EKR}. We briefly describe this construction: for some $1\le i<j\le n$ apply $S\mapsto (S\setminus\{j\})\cup\{i\}$ for all $j\in S\in K$, $i\notin S$ to obtain $K'$, which is also a simplicial complex, with the same $f$-vector. Repeat this process as long as possible to obtain a shifted simplicial complex which we denote $\Delta^c(K)$. Note that $\Delta^c(K)$ depends on the order of choices of pairs $i<j$. 

Note that a shifted simplicial complex $K$ is homotopy equivalent to a wedge of spheres, where the number of $i$-dimensional spheres in this wedge is $|\{S\in K_i:S\cup\{1\}\notin K\}|$. In particular, 
$\beta_i(K,\mathbb L)=|\{S\in K_i:S\cup\{1\}\notin K\}|$ for every field $\mathbb L$. 

A {\itshape $k$-uniform hypergraph} on $n$ vertices is a collection of subsets of $\binom{[n]}{k}$. In particular, 
$k$-uniform hypergraphs are in bijection with pure $(k-1)$-dimensional simplicial complexes, by closing down to subsets $H\mapsto K(H)$.
We denote by $H_i, f_i(H), \beta_i(H)$ the $i$-skeleton, number of $i$-faces, and $i$-th Betti number of the corresponding simplicial complex $K(H)$, resp.

Let $<_{\text{lex}}$ be the lexicographic order on equal sized subsets of $\mathbb N$, i.e. $S<_{\text{lex}}T$ iff $\min(S\triangle T)\in S$. Throughout the paper, we will use a convention that writing $\{i_1,i_2,\ldots,i_k\}$ means an ordered set, namely $i_1<i_2<\ldots<i_k$, unless specified otherwise.

\subsection{Algebraic shifting} \label{subsection: algebraic shifting}
The presentation in this section is based mainly on \cite{K Alg Sh, Nevo phd}.

\textbf{Exterior shifting.} 
Let $\mathbb F$ be a field and let $\mathbb L$ be a field extension of $\mathbb F$ of transcendental degree $\ge n^2$. (e.g. $\mathbb F=\mathbb Q$ and $\mathbb L=\mathbb R$). Let $V$ be an $n$-dimensional vector space over $\mathbb L$ with a basis $\{e_1,\ldots,e_n\}$. Let $\bigwedge V$ be the graded exterior algebra over $V$. Denote $e_S=e_{s_1}\wedge\ldots\wedge e_{s_k}$ where $S=\{s_1<\ldots<s_k\}$. Then $\{e_S:S\in\binom{[n]}{k}\}$ is a basis for $\bigwedge^kV$, the {\itshape $k$-th exterior power} of $V$. Note that as $K$ is a simplicial complex, the ideal $(e_S:S\notin K)$ of $\bigwedge V$ and the vector subspace $\text{span}\{e_S:S\notin K\}$ of $\bigwedge V$ consist of the same set of elements in $\bigwedge V$. Define the exterior algebra of $K$ by

$$\bigwedge K=(\bigwedge V)/(e_S:S\notin K).$$

Let $\{f_1,\ldots,f_n\}$ be a basis of $V$, generic over $\mathbb F$ with respect to $\{e_1,\ldots,e_n\}$, which means that the entries of the corresponding transition matrix $X$ ($e_iX=f_i$ for all $i$) are algebraically independent over $\mathbb F$. Let $\tilde{f}_S$ be the image of $f_S\in\bigwedge V$ in $\bigwedge(K)$. 
Define 
$$\Delta^e(K)=\{S:\tilde{f}_S\notin\text{span}\{\tilde{f}_{S'}:S'<_{\text{lex}}S\}\}$$
to be the {\itshape exterior shifting} of $K$.

Here is an equivalent definition of exterior shifting. Let $X=(x_{ij})_{1\le i\le n, 1\le j\le n}$ be an $n$ by $n$ matrix over $\mathbb L$ which is generic over $\mathbb F$, i.e. a matrix whose entries are algebraically independent over 
$\mathbb F$. Let $X\left(\binom{[n]}{k};\binom{[n]}{k}\right)$ be the {\itshape $k$-th compound matrix} of $X$, i.e., the $\binom{n}{k}$ by $\binom{n}{k}$ matrix of $k$ by $k$ minors of $X$. Assume that the rows and the columns of $X\left(\binom{[n]}{k};\binom{[n]}{k}\right)$ are ordered lexicographically.

Given a collection $H$ of $k$-subsets of $[n]$ with $|H|=m$, let $X\left(H;\binom{[n]}{k}\right)$ be the $m$ by $\binom{n}{k}$ submatrix of $X\left(\binom{[n]}{k};\binom{[n]}{k}\right)$ whose rows correspond to the $k$-sets in $H$. Now choose a basis of columns for the column-space of $X\left(H;\binom{[n]}{k}\right)$ in the greedy way w.r.t. the lexicographic order: by simply taking those columns which are not spanned by previous columns in the lexicographic order. Starting from a generic over $\mathbb F$ matrix $X$ define $\Delta^e(H)$ to be the family of sets which are the indices of the chosen columns. For a simplicial complex $K$ define $\Delta(K)=\cup_i\Delta(K_{i})$, where $K_{i}$ is an $i$-dimensional skeleton of $K$.

The equivalence of the definitions follows from the fact that $f_S=\sum X_{ST}e_T$, where $X_{ST}$ is the $k$ by $k$ minor of $X$ with rows and columns corresponding to the elements of $S$ and $T$ respectively. 

\textbf{Symmetric shifting.} Let $K$ be a simplicial complex and let $R(K)$ be its Stanley-Reisner ring (face ring):
$$R(K)=\mathbb L[x_1,\ldots,x_n]/I_K,$$
where $I_K$ is the homogeneous ideal generated by the monomials whose support is not in $K$, i.e. monomials $x_{i_1}\cdot x_{i_2}\cdot \ldots\cdot x_{i_r}$ where $\{i_1,i_2,\ldots,i_r\}\notin K$. Let $y_1,\ldots,y_n$ be generic over $\mathbb F$ linear combinations of $x_1,\ldots,x_n$. 
All monomials in $y_i$'s span $R(K)$. We now construct the basis $\text{GIN}(K)$ of monomials in the variables $y_i$ in the greedy way w.r.t. the lexicographic order. That is, a monomial $m$ belongs to $\text{GIN}(K)$ if and only if its image $\tilde{m}$ in $R(K)$ is not a linear combination of images of monomials which are lexicographically smaller, i.e.
$$\text{GIN}(K)=\{m:\tilde{m}\notin\text{span}_k\{\tilde{m'}:\deg(m')=deg(m),m'<_{\text{lex}}m\}\},$$
where $\prod\limits_{i=1}^n y_i^{a_i}<_{\text{lex}}\prod\limits_{i=1}^n y_i^{b_i}$ iff for $j=\min\{i:a_i\ne b_i\}\ a_j>b_j$. 

For example,
$y_1^2<_{\text{lex}}y_1y_2<_{\text{lex}}y_1y_3<_{\text{lex}}\ldots<_{\text{lex}}y_1y_n<_{\text{lex}}y_2^2<_{\text{lex}}\ldots$.

The combinatorial information in $\text{GIN}(K)$ is redundant: if $m\in \text{GIN}(K)$ is of degree $i\le \dim(K)$ then $y_1m,\ldots,y_im$ are also in $\text{GIN}(K)$. Thus, $\text{GIN}(K)$ can be reconstructed from its monomials of the form $m=y_{i_1}\cdot y_{i_2}\cdot\ldots\cdot y_{i_r}$ where $r\le i_1\le i_2\le\ldots\le i_r$, $r\le \dim(K)+1$. Denote this set by $gin(K)$, and associate the set $S_{un}(m)=\{i_1-r+1,i_2-r+2,\ldots,i_r\}$ with such a monomial $m$. Denote 
$$\Delta^s(K)=\cup\{S_{un}(m):m\in \text{gin}(K)\}$$
to be the {\itshape symmetric shifting} of $K$.

Given a simplicial complex $K$, we will use the notation $\Delta(K)$ in the cases where a statement is satisfied for both $\Delta^e(K)$ and $\Delta^s(K)$.

Both exterior and symmetric shiftings are canonical in the sense that they do not depend on the choice of a generic basis and for a permutation $\pi:[n]\to[n]$ the induced simplicial complex $K$ satisfies $\Delta(\pi(K))=\Delta(K)$. The following Theorem summarizes other properties of algebraic shifting (from now on we will consider Betti numbers with coefficients in $\mathbb L$):

\begin{theorem} \textup{\cite[Theorems 2.1, 2.2 and 3.2]{K Alg Sh}} \label{Basic properties of algebraic shifting}
Let $K$ and $L$ be simplicial complexes, and let $\Delta$ denote algebraic shifting (exterior or symmetric). Then:
\begin{enumerate}
    \item $\Delta(K)$ is a simplicial complex.
    \item $\Delta(K)=\Delta(L)$ for $L$ combinatorially isomorphic to $K$, see \textup{\cite[Claim 1]{BK}}.
    \item $f(K)=f(\Delta(K))$, see \textup{\cite[Theorem 3.1]{BK}}.
    \item $\beta(K)=\beta(\Delta(K))$, see \textup{\cite[Theorem 3.1]{BK}}.
    \item $\Delta(K)$ is shifted, see \textup{\cite[Theorem 3.1]{BK}, \cite[Remark (4)]{K Symm}}.
    \item If $K$ is shifted then $\Delta(K)=K$, see \textup{\cite[Claim 1]{BK}, \cite[Remark (4)]{K Symm}}.
    \item If $L\subseteq K$, then $\Delta(L)\subseteq \Delta(K)$.
    \item $\Delta(Cone(K))=Cone(\Delta(K))$, see \textup{ \cite[Corollary 5.4]{Nevo 2005}}, \textup{\cite[Lemma 3.3]{BNT}}.
    \item $\Delta(K)$ depends only on the characteristic of $\mathbb F$.
\end{enumerate}
\end{theorem}

\subsection{Matroids}
Throughout this section we follow \cite{Oxley}.

\begin{definition}
    A {\itshape matroid} $M$ (on $E$) is an ordered pair $(E,\mathcal{I})$ consisting of a finite set $E$ and a collection $\mathcal{I}$ of subsets of $E$ having the following three properties:
    \begin{enumerate}
        \item[\bf{(I1)}] $\emptyset\in\mathcal{I}$.
        \item[\bf{(I2)}] If $I\in\mathcal{I}$ and $I'\subseteq I$, then $I'\in\mathcal{I}$.
        \item[\bf{(I3)}] If $I_1$ and $I_2$ are in $\mathcal{I}$ and $|I_1|<|I_2|$, then there is an element $e$ of $I_2\setminus I_1$ such that $I_1\cup e\in\mathcal{I}$.
    \end{enumerate}
\end{definition}
The members of $\mathcal{I}$ are the {\itshape independent sets} of $M$, and $E$ is the {\itshape ground set} of $M$. The maximal independent sets are called {\itshape bases}.
Here is an equivalent definition of a matroid in terms of bases:
\begin{definition}
    A {\itshape matroid} $M$ is an ordered pair $(E,\mathcal{B})$ consisting of a finite set $E$ and a collection $\mathcal{B}$ of subsets of $E$ having the following two properties:
    \begin{enumerate}
        \item[\bf{(B1)}] $\mathcal{B}$ is non-empty.
        \item[\bf{(B2)}] If $B_1$ and $B_2$ are members of $\mathcal{B}$ and $e_1\in B_1\setminus B_2$, then there is an element $e_2\in B_2\setminus B_1$ such that $(B_1\setminus e_1)\cup e_2\in \mathcal{B}$.
    \end{enumerate}
    The property \textbf{(B2)} is called the {\itshape exchange axiom}.
\end{definition}
\begin{example}
    Any finite collection of vectors is a ground set for a matroid whose independent sets are linearly independent subsets of this collection and bases are linear bases of the linear span of this collection.
\end{example}

\section{Matroids arising from exterior shifting} \label{sec: Matroids arising from exterior shifting}
Let $H$ be a shifted $k$-uniform hypergraph on $[n]$. 
Denote by
$$M(H)=\left\{G\subseteq \binom{[n]}k|\ \Delta^e(G)=H\right\}$$
the set of all $k$-uniform hypergraphs on $[n]$ whose exterior shifting equals $H$.
\begin{definition}
Let $H$ be a shifted $k$-uniform hypergraph on $[n]$. $H$ is {\itshape matroidal} if $M(H)$ is the set of bases of a matroid on $\binom{[n]}{k}$. 
\end{definition}

\textbf{Kalai's construction:} Let $X=(x_{ij})_{1\le i\le n,1\le j\le n}$ be an $n\times n$-matrix over $\mathbb L$.

For $S,T\in\binom{[n]}{k}$ let $X_{ST}$ be the minor which corresponds to the $k$ rows of $X$ indexed by the elements of $S$ and the $k$ columns of $X$ indexed by the elements of $T$. That is, if $S=\{i_1,\ldots,i_k\}$ and $T=\{j_1,\ldots,j_k\}$, then $X_{ST}=\det(x_{i_sj_t})_{1\le s\le k, 1\le t\le k}$.

Let $B$ and $H$ be subsets of $\binom{[n]}{k}$. Denote by $X(B,H)$ the matrix $(X_{ST})_{S\in B,T\in H}$. Here, the rows and the columns of $X(B,H)$ are ordered by the lexicographic order on $\binom{[n]}{k}$. 

The following definition appeared in \cite[Section 4]{K Symm} and allows to construct a symmetric matroid for any given $H\subseteq\binom{[n]}k$.

Let $X=(x_{ij})_{1\le i,j\le n}$ be an $n\times n$-matrix over $\mathbb L$ which is generic
over $\mathbb F$. For a $k$-uniform hypergraph $H\subseteq\binom{[n]}k$ define a matroid $M'(H)$ on $\binom{[n]}k$ as follows: $B\subseteq\binom{[n]}k$ is independent in $M'(H)$ iff the rows of $X(B,H)$ are linearly independent. 
\begin{remark} \cite[Section 4, Remark (2)]{K Symm} \label{Remark: Kalai's construction for hypperconnectivity}
 For $H=\{S\subseteq\binom{[n]}2:S\cap[k]\ne\emptyset\},\ M'(H)$  is the $k$-hyperconnectivity matroid.  
\end{remark}
In fact, a more general statement is straightforward from the construction above. We call a subset $H\subseteq\binom{[n]}{k}$ an {\itshape initial lex-segment} 
if $H=\left\{T'\in\binom{[n]}{k}|\ T'\le_{\text{lex}}T\right\}$ for some $T\in\binom{[n]}{k}$.
\begin{corollary}
\label{Initial segment is a matroid}
 Let $H\subseteq\binom{[n]}k$ be a $k$-uniform hypergraph whose hyperedges form an initial lex-segment. Then $H$ is matroidal.
\end{corollary}
\begin{proof}
We will verify that $M(H)$ is the set of bases of the matroid $M(H')$, or $M(H)=M'(H)$ for short.
Let $B\in M(H)$, then $H$ is a column basis chosen in the greedy way for the matrix $X(B,\binom{[n]}{k})$. Thus, the columns of the matrix $X(B,H)$ are linearly independent, hence $B\in M'(H)$.
 
For a basis $B\in M'(H)$ let us compute the exterior shifting of $B$: as the columns in $X(B,H)$ are linearly independent, and as for exterior shifting the basis of columns in the compound matrix is chosen in the greedy way,  
$\Delta^e(B)$ is given by some set $G$ corresponding to $|H|$ columns lexicographically smaller than or equal to $H$. However, since $H$ is an initial lex-segment, we conclude that this set $G$ coincides with $H$, namely  $\Delta^e(B)=H$. 
Thus, $M(H)=M'(H)$ as desired.
\end{proof}

\subsection{Preparatory lemmas} \label{sec: Preparatory lemmas}
In this section we present a series of lemmas and reductions which we will use in proofs of Theorems \ref{thm: main thm 2 and 3} and \ref{main theorem}. 
In order to show that a shifted but not initial lex-segment $k$-uniform hypergraph $H$, satisfying some \emph{extra condition}, is not matroidal, we will choose a permutation $\pi$ on $[n]$ and an element $e_1\in H$ for which we then show that for every $e_2\in\pi(H)\setminus H$ holds $\Delta(H\setminus e_1)\cup e_2)\ne H$. Thus, using  Theorem~\ref{Basic properties of algebraic shifting}(2), we will conclude that $M(H)$ violates the exchange axiom.

Here is the first such \emph{extra condition}, giving the first reduction.

\begin{lemma}[Isolated vertex] \label{lemma: partial star with an additional edge 2,3 is not a matroid}
 Let $H\subseteq\binom{[n]}{k}$ be a shifted $k$-uniform hypergraph such that $n$ is an isolated vertex and $\{2,3,\ldots,k+1\}\in H$. 
 Then $H$ is not matroidal. 
\end{lemma}
\begin{proof}
    Consider two $k$-uniform hypergraphs $H_1=H$ and $H_2$ obtained from $H$ by flipping the vertex $1$ with the vertex $n$ (see Fig. \ref{graphs which do not satisfy the exchange exiom} for illustration). By its structure $H_1$ contains the facets of the boundary complex $\partial \sigma$ of the simplex $\sigma=\{1,2,\ldots,k+1\}$. In order to show that the pair $H_1, H_2$ violates the exchange axiom we have to find $e_1\in H_1\setminus H_2$ such that for every $e_2\in H_2\setminus H_1$ we have that
    $(H_1\setminus e_1)\cup e_2\notin M(H)$. Let us take as $e_1$ an arbitrary facet of 
    $\partial \sigma$
    which contains $1$. Such $e_1$ indeed lies in $H_1\setminus H_2$ since $1$ is an isolated vertex in $H_2$. Note that $\beta_{k-1}(H_1\setminus e_1)<\beta_{k-1}(H_1)$. As $e_2$ we can take only hyperedges containing the vertex $n$, which is an isolated vertex in $H_1$. Thus, $e_2$ 
    does not belong to any $(k-1)$-homology cycle in $(H_1\setminus e_1)\cup e_2$. Thus, $\beta_{k-1}((H_1\setminus e_1)\cup e_2)=\beta_{k-1}(H_1\setminus e_1)<\beta_{k-1}(H_1)$. However, since algebraic shifting preserves Betti numbers, see Theorem \ref{Basic properties of algebraic shifting}(4), we have $\beta_{k-1}(\Delta((H_1\setminus e_1)\cup e_2))=\beta_{k-1}((H_1\setminus e_1)\cup e_2)=\beta_{k-1}(H_1\setminus e_1)<\beta_{k-1}(H_1)=\beta_{k-1}(H)$. Thus, $((H_1\setminus e_1)\cup e_2)\notin M(H)$.
\end{proof}

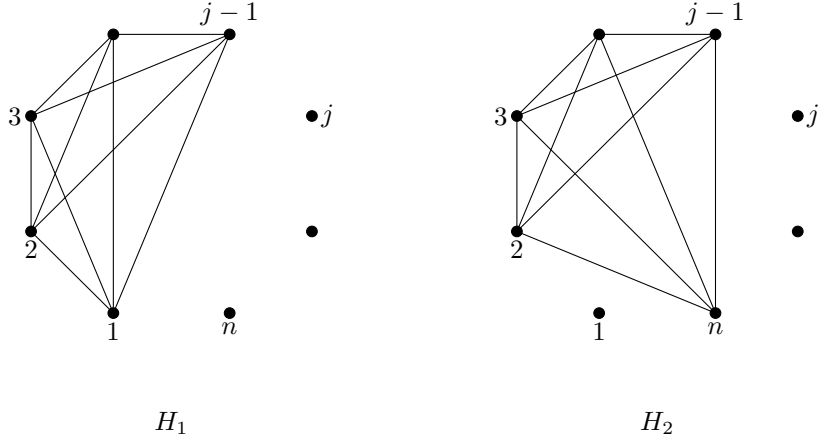
\begin{figure}
\centering
\begin{tikzpicture}
    \filldraw (22.5:2) node [right] {$j$} circle (2pt);
    \filldraw (67.5:2) node [above] {$j-1$} circle (2pt);
    \filldraw (112.5:2) node [above] {} circle (2pt);
    \filldraw (157.5:2) node [left] {$3$} circle (2pt);
    \filldraw (-22.5:2) node [right] {} circle (2pt);
    \filldraw (-67.5:2) node [below] {$n$} circle (2pt);
    \filldraw (-157.5:2) node [below] {$2$} circle (2pt);
    \filldraw (-112.5:2) node [below] {$1$} circle (2pt);

    \draw (-112.5:2) -- node [above] {}(-157.5:2);
    \draw (-112.5:2) -- node [above] {}(157.5:2);
    \draw (-112.5:2) -- node [above] {}(112.5:2);
    \draw (-112.5:2) -- node [above] {}(67.5:2);
    \draw (-157.5:2) -- node [above] {}(157.5:2);
    \draw (-157.5:2) -- node [above] {}(112.5:2);
    \draw (-157.5:2) -- node [above] {}(67.5:2);
    \draw (157.5:2) -- node [above] {}(112.5:2);
    \draw (157.5:2) -- node [above] {}(67.5:2);
    \draw (112.5:2) -- node [above] {}(67.5:2);
    
    \draw (0,-3.3) node {$H_1$};
\end{tikzpicture}\phantom{aaaaaaaaaa}
\begin{tikzpicture}

    \filldraw (22.5:2) node [right] {$j$} circle (2pt);
    \filldraw (67.5:2) node [above] {$j-1$} circle (2pt);
    \filldraw (112.5:2) node [above] {} circle (2pt);
    \filldraw (157.5:2) node [left] {$3$} circle (2pt);
    \filldraw (-22.5:2) node [right] {} circle (2pt);
    \filldraw (-67.5:2) node [below] {$n$} circle (2pt);
    \filldraw (-157.5:2) node [below] {$2$} circle (2pt);
    \filldraw (-112.5:2) node [below] {$1$} circle (2pt);

    \draw (-67.5:2) -- node [above] {}(-157.5:2);
    \draw (-67.5:2) -- node [above] {}(157.5:2);
    \draw (-67.5:2) -- node [above] {}(112.5:2);
    \draw (-67.5:2) -- node [above] {}(67.5:2);
    \draw (-157.5:2) -- node [above] {}(157.5:2);
    \draw (-157.5:2) -- node [above] {}(112.5:2);
    \draw (-157.5:2) -- node [above] {}(67.5:2);
    \draw (157.5:2) -- node [above] {}(112.5:2);
    \draw (157.5:2) -- node [above] {}(67.5:2);
    \draw (112.5:2) -- node [above] {}(67.5:2);

    \draw (0,-3.3) node {$H_2$};
\end{tikzpicture}
\caption{The graphs $H_1$ and $H_2$ violate the exchange axiom}
    \label{graphs which do not satisfy the exchange exiom}
\end{figure}
For a shifted $k$-uniform hypergraph $H$ on $[n]$ let its {\itshape cone} be the $(k+1)$-uniform hypergraph  $\text{Cone}(H)=\{\{0\}\cup F:\ F\in H\}$ where $0<1$ for the definition of the lexicographic order over the larger vertex set $\{0\}\cup [n]$.

\begin{lemma}[Cone] \label{coning preserve property to not be a matroid}
    Let $H\subseteq\binom{[n]}{k}$ be a shifted $k$-uniform hypergraph. Suppose that $H$ is not matroidal. Then the $(k+1)$-uniform shifted hypergraph $\text{Cone}(H)\subseteq\binom{\{0\}\cup[n]}{k+1}$ is also not matroidal.
\end{lemma}
\begin{proof}
     $H$ is not matroidal, thus there is a pair of hypergraphs in $M(H)$, denoted $H_1$ and $H_2$, that violates the exchange axiom among the elements of $M(H)$.
Recall that by Theorem \ref{Basic properties of algebraic shifting}(8), for a $k$-uniform hypergraph $G$,  
    $\Delta(\text{Cone}(G))=\text{Cone}(\Delta(G))$. 
    In particular, the hypergraphs $\{0\}*H_1$ and $\{0\}*H_2$ belong to $M(\{0\}*H)$. 
    
    We now prove that $\{0\}*H_1$ and $\{0\}*H_2$ form a violating pair for the exchange axiom among the elements of $M(\{0\}*H)$. Since $H_1$ and $H_2$ violate the exchange axiom among the elements of $M(H)$, there exists $e_1\in H_1\setminus H_2$ such that for every $e_2\in H_2\setminus H_1$ we have that $(H_1\setminus e_1)\cup e_2\notin M(H)$. 
    Since every element in $(\{0\}*H_2)\setminus (\{0\}*H_1)$ has a form $e_2\cup\{0\}$ for some $e_2\in H_2\setminus H_1$, it suffices to show that $\left(\{0\}*H_1\setminus \{e_1\cup \{0\}\}\right)\cup \{e_2\cup\{0\}\}\notin M(\{0\}*H)$.

    By Theorem \ref{Basic properties of algebraic shifting}(8) again,
    $$\Delta((\{0\}*H_1\setminus(e_1\cup\{0\}))\cup(e_2\cup\{0\}))=\{0\}*\Delta(((H_1\setminus e_1)\cup e_2)).$$
    On the other hand, since $\Delta((H\setminus e_1)\cup e_2)\ne H$, we get $\{0\}*\Delta(((H_1\setminus e_1)\cup e_2))\ne \{0\}*H$.
    \end{proof}
We use this Cone Lemma to prove the
lemma below, which in turn serves in the inductive step in the proof of Theorem \ref{thm: main thm 2 and 3}.  
While the base case is proven differently in the $2$-uniform case and the $3$-uniform case of Theorem \ref{thm: main thm 2 and 3}, 
 the induction step is proved uniformly in both cases, so it is separated in the lemma below. 
Denote $\text{st}(1):=\{S\in\binom{[n]}{k}|\ 1\in S\}$. 

\begin{lemma}[Reduction] \label{lemma: reduction}
Assume that for every 
    shifted $k$-uniform hypergraph $H$ whose hyperedges do not form an initial lex-segment and additionally satisfies the two conditions below, it holds that $H$ is not matroidal.

    $(i)$ $\{2,3,\ldots,k,k+1\}$ is a hyperedge of $H$.

    $(ii)$ There exists a subset $\{1,i_2,\ldots,i_k\}\notin H$ with $i_2>2.$

    Then, for every shifted $k$-uniform hypergraph $H$ whose hyperedges do not form an initial lex-segment it holds that $H$ is not matroidal.
\end{lemma}
\begin{proof} 
Let $H$ be a shifted $k$-uniform hypergraph on $[n]$ whose hyperedges do not form an initial lex-segment.

Assume that $H$ does not satisfy (i). We prove by induction on $k$ that $H$ is not matroidal.

As $\{2,3,\ldots,k,k+1\}$ is not in $H$ and $H$ is shifted, 
  $H$ has no hyperedge of the form $\{i_1,i_2,\ldots,i_k\}$ with $i_1\ge 2$, i.e. $H=\{1\}* H'$ where $H'$ is a $(k-1)$-uniform shifted hypergraph (on $[n]\setminus \{1\}$) whose hyperedges do not form an initial lex-segment. Thus, by the induction hypothesis $H'$ is not matroidal. Hence by Lemma \ref{coning preserve property to not be a matroid} $H$ is not matroidal. To conclude, we may assume that the set $\{2,3,\ldots,k,k+1\}$ is a hyperedge of $H$.
 
Assume that $H$ does not satisfy (ii). 
First note that this means that for every $1\in S\in\binom{[n]}{k}$ it holds that $S\in H$. Indeed, 
if not then there exists $\{1,i_2,\ldots,i_k\}\notin H$. However, $\{1,i_2,\ldots,i_k\}\le_p\{1,n-k+2,\ldots,n-1,n\}$, and as $H$ is shifted we obtain $\{1,n-k+2,\ldots,n-1,n\}\notin H$. 
Negating (ii) gives
$2=n-k+2$, hence $n=k$ and $H=\{[n]\}$ which is an initial lex-segment, a contradiction. 

We now prove by induction on $n$ that $H$ is not matroidal. 
By assumption $H\setminus \text{st}(1)$ is not an initial 
lex-segment (on $[n]\setminus\{1\}$), and by
the induction hypothesis it is not matroidal. 
Then there is a pair of $k$-uniform hypergraphs $G_1,G_2$ on $\{2,\ldots,n\}$ violating the exchange axiom among the elements of $M(H\setminus \text{st}(1))$, i.e. there exists $e_1\in G_1\setminus G_2$ such that for every $e_2\in G_2\setminus G_1$ we have that $\Delta((G_1\setminus e_1)\cup e_2))\ne H\setminus\text{st}(1)$. 
(Here we shift w.r.t. the lexicographic order on $k$-subsets of  $\{2,\ldots,n\}$.)

We now prove that 
$\text{st}(1)\cup G_1, \text{st}(1)\cup G_2\in M(H)$, and they form a violating pair for the exchange axiom among the elements of $M(H)$.

On the one hand, $\text{st}(1)\subseteq \text{st}(1)\cup G_1$. Thus, by shiftedness of $\text{st}(1)$ and the containment preserving property, Theorem \ref{Basic properties of algebraic shifting}(7), we have that 
\begin{equation} \label{eq: proof 0}
\text{st}(1)=\Delta(\text{st}(1))\subseteq \Delta(\text{st}(1)\cup G_1).
\end{equation}

On the other hand, 
$K(\{1\}*G_1)_{\leq k-1}\subseteq K(\text{st}(1)\cup G_1).$
Thus, since algebraic shifting is defined degreewise and by Theorem \ref{Basic properties of algebraic shifting}(7) again,
\begin{equation} \label{eq: proof 1}
\Delta\left(K\left(\{1\}*G_1\right)_{k-1}\right)\subseteq\Delta\left(\text{st}(1)\cup G_1\right).
\end{equation}
Since algebraic shifting is defined degreewise,
\begin{equation} \label{eq: proof 2}
\Delta\left(K\left(\{1\}*G_1\right)_{k-1}\right)=\left(\Delta\left(K(\{1\}*G_1\right)\right)_{k-1}.
\end{equation}
By the Cone property Theorem \ref{Basic properties of algebraic shifting}(8) applied to the $(k-1)$-skeleton  we have:
\begin{equation} \label{eq: proof 3}
\left(\Delta\left(K(\{1\}*G_1)\right)\right)_{k-1}=\left(\{1\}*\Delta(K(G_1)\right))_{k-1}.
\end{equation}
Substituting \eqref{eq: proof 3} in \eqref{eq: proof 2}, and then \eqref{eq: proof 2} in \eqref{eq: proof 1} we obtain:
\begin{equation} \label{eq: proof 4}
    \left(\{1\}*\Delta(K(G_1))\right)_{k-1}\subseteq\Delta\left(\text{st}(1)\cup G_1\right).
\end{equation}
Note that since $\Delta(G_1)=H\setminus\text{st}(1)$, we get
\begin{equation} \label{eq: proof 5}
    \left(\{1\}*\left(K(H\setminus \text{st}(1)\right)\right))_{k-1}\subseteq\Delta\left(\text{st}(1)\cup G_1\right).
\end{equation}

Since algebraic shifting preserves the $f$-vector, the conditions \eqref{eq: proof 0} and \eqref{eq: proof 5} determine the algebraic shifting of $\text{st}(1)\cup G_1$, which is then equal to:
$$\Delta\left(\text{st}(1)\cup G_1\right)=H.$$

Similarly for $\text{st}(1)\cup G_2$.

We now show  
that 
the pair $\text{st}(1)\cup G_1, \text{st}(1)\cup G_2$ violates the exchange axiom among the elements of $M(H)$. For every $e_2\in (\text{st}(1)\cup G_2)\setminus(\text{st}(1)\cup G_1)$ we have that 

$$((\text{st}(1)\cup G_1)\setminus e_1)\cup e_2=\text{st}(1)\cup\left(\left(G_1\setminus e_1\right)\cup e_2\right),$$
because $e_1\in G_1\setminus G_2$ and $e_2\in G_2\setminus G_1$, in particular $e_2\notin\text{st}(1)$.

Substituting in \eqref{eq: proof 4} 
the hypergraph 
$\left(G_1\setminus e_1\right)\cup e_2$ instead of $G_1$ we get:
$$\left(\{1\}*\Delta\left(K(\left(G_1\setminus e_1\right)\cup e_2\right)\right))_{k-1}\subseteq\Delta\left(\text{st}(1)\cup \left(\left(G_1\setminus e_1\right)\cup e_2\right)\right).$$
Then, since $\Delta\left(\left(G_1\setminus e_1\right)\cup e_2\right)\ne H\setminus\text{st}(1)$, we conclude that
$$\Delta\left(\text{st}(1)\cup \left(\left(G_1\setminus e_1\right)\cup e_2\right)\right)\ne H.$$
\end{proof}

\begin{lemma}[Homology] \label{lemma: homology}
Let $H\subset\binom{[n]}{k}$ be a 
$k$-uniform 
shifted hypergraph such that a subset $\{1,i_2,\ldots,i_k\}\notin H$ and all facets of a $k$-simplex boundary complex $\partial T$ are hyperedges of $H$. Then, for every facet $e_1$ of $\partial T$ and every $e_2\in\binom{[n]}{k}$ such that $\{i_2,\ldots,i_k\}\subset e_2$, we have $\Delta((H\setminus e_1)\cup e_2)\ne H$.     
\end{lemma}

\begin{proof}
The deletion of a hyperedge $e_1$ from $H$ decreases $\beta_{k-1}$, i.e. $\beta_{k-1}(H\setminus e_1)<\beta_{k-1}(H)$. Since $\{1,i_2,\ldots,i_k\}$ is not a hyperedge of $H$, 
then by shiftedness 
there are no hyperedges in $H$ which contain the set $\{i_2,\ldots,i_k\}$. 
Thus, 
adding $e_2$ 
completes
no $(k-1)$-homology cycle, and 
thus $\beta_{k-1}((H\setminus e_1)\cup e_2)=\beta_{k-1}(H\setminus e_1)<\beta_{k-1}(H)$. Hence, by Theorem \ref{Basic properties of algebraic shifting}(4), $\Delta((H\setminus e_1)\cup e_2)\ne\Delta(H)=H$.
\end{proof}

\begin{lemma}[Small $e_2$]  \label{lemma: e_2<max}
Let $H\subseteq\binom{[n]}{k}$ be a shifted $k$-uniform hypergraph and let $e_1\in H$. Then for every $e_2\in \binom{[n]}{k}$ such that $e_2\notin H$ and $e_2<_{\text{lex}}e_1$, $\Delta^e((H\setminus e_1)\cup e_2)\neq H$.
\end{lemma}
\begin{proof}
    Denote $H':=(H\setminus e_1)\cup e_2$. Consider a compound matrix $X(H',H')$. By the same argument as in \cite[Proposition 4.2]{K Symm} we get that $\det X(H',H')\ne 0$. Thus, the greedy basis for the column space of the matrix $X(H',\binom{[n]}{k})$ is lexicographically not greater than $H'$. Hence $\Delta^e(H')\le_{\text{lex}}H'$. On the other hand, $H'<_{\text{lex}}H$ since $e_2<_{\text{lex}}e_1$. We conclude that $\Delta^e((H\setminus e_1)\cup e_2)\neq \Delta^e(H)=H$.  
\end{proof}

Denote by $m_H$ the maximum  element of $H$ w.r.t. the lexicographic order. If $e_1=m_H$, then the assertion of Lemma \ref{lemma: e_2<max} holds also for symmetric shifting. The proof below works for both exterior and symmetric shifting.
\begin{lemma}[Case $e_2<_{\text{lex}}m_H$] 
    Let $H\subseteq\binom{[n]}{k}$ be a shifted $k$-uniform hypergraph. 
    Then for every $e_2\in \binom{[n]}{k}$ such that $e_2\notin H$ and $e_2<_{\text{lex}}m_H$, $\Delta((H\setminus m_H)\cup e_2)\neq H$.
\end{lemma}
\begin{proof}
    Consider the hypergraph $H'$ given by the initial lex-segment consisting of all elements up to $m_H$ not including $m_H$. $H'$ is shifted, thus by Theorem \ref{Basic properties of algebraic shifting}(6) $\Delta(H')=H'$. On the other hand, $ ((H\setminus m_H)\cup e_2) \subset H'$, and thus by the containment preserving property Theorem \ref{Basic properties of algebraic shifting}(7), we have that $\Delta((H\setminus m_H)\cup e_2)\subset\Delta(H')$. Note that $m_H\notin H'=\Delta(H')$ and thus   
    $m_H\notin \Delta((H\setminus m_H)\cup e_2)$. However, $m_H\in H=\Delta(H)$. 
\end{proof}

\begin{lemma}[Induced subgraph] \label{lemma: induced subgraph}
    Let $H\subseteq\binom{[n]}{k}$ be a shifted $k$-uniform hypergraph, let $e_1\in H$ and let $e_2\in\binom{[n]}{k}$ such that $\max e_2<\max e_1$. Then, $\Delta((H\setminus e_1)\cup e_2)\ne H$.
\end{lemma}
\begin{proof}
    Denote $t=\max e_2$. Consider the induced $k$-uniform hypergraph $H'\in\binom{[t]}{k}$ given by
    $$H':=\left\{I\in H|\ I\in\binom{[t]}{k}\right\}.$$
    Note that since $H$ is shifted, then so is $H'$. 

    On the one hand, $H'\cup e_2\subseteq (H\setminus e_1)\cup e_2$, and thus by the containment preserving property, Theorem~\ref{Basic properties of algebraic shifting}(7), $\Delta(H'\cup e_2)\subseteq \Delta((H\setminus e_1)\cup e_2)$. On the other hand, since $H'$ is shifted, by Theorem \ref{Basic properties of algebraic shifting}(6) we have  $H'=\Delta(H')\subset\Delta(H'\cup e_2)$. Thus, $\Delta(H'\cup e_2)=H'\cup \{J\}$, where $J\in\binom{[t]}{k}$. 
    We conclude that $J\in \Delta((H\setminus e_1)\cup e_2)$ while  $J\notin H$.
\end{proof}

\begin{lemma}[Case $e_2$ is almost small.] \label{lemma: compound matrix with s}
Let $H\in\binom{[n]}{k}$ be a shifted $k$-uniform hypergraph. Let $e_1\in H$ and let $e_2\in\binom{[n]}{k}$ such that $e_2\notin H$. Let $S\in\binom{[n]}{k}$ such that $S<_{\text{lex}}e_1$, $S\notin H$ and $|e_2\setminus S|=|S\setminus e_2|=1$. Then $\Delta^e((H\setminus e_1)\cup e_2)\ne H$.
\end{lemma}

\begin{proof}
Denote $e_2=\{l_1<l_2<\ldots<l_k\}$, $e_2\setminus S=\{l_i\}$ and $S\setminus e_2=\{y\}$.
For convenience order $S=\{l_1,\ldots,l_{i-1},y,l_{i+1},\ldots l_k\}$. 
Consider the compound matrix $X((H\setminus e_1)\cup e_2;(H\setminus e_1)\cup S)$. Contrary to the usual notations order its rows and columns in the following way: rows and columns corresponding to $H\setminus e_1$ are ordered in the lexicographic order, the row corresponding to $e_2$ is the last row, and the column corresponding to $S$ is the last column, see Fig. \ref{pic:cofactor matrix with S}. 

Note that the reorderings made above do not affect the invertibility of $X((H\setminus e_1)\cup e_2;(H\setminus e_1)\cup S)$ since reordering of the subset $S$ leads to the same $\pm 1$ sign change in front of all the entries in the column corresponding to $S$, which preserves the linear span of the columns, and reordering rows and columns also only changes the sign of the determinant.

Next we prove that the following determinant is nonzero:

\begin{equation} \label{eq: def of det}
    \det X((H\setminus e_1)\cup e_2;(H\setminus e_1)\cup S)=\sum\limits_{\sigma\in S_{f_{k-1}(H)}}(-1)^{sgn (\sigma)}X_{T_1\sigma(T_1)}\cdot\ldots\cdot X_{T_{f_{k-1}(H)}\sigma(T_{f_{k-1}(H)})},
    \end{equation}
    where $S_m$ is a permutation group of the set $[m]$.

Note that each factor $X_{T_i\sigma(T_i)}$, where $1\le i\le f_{k-1}(H)$, in the expression above is a certain determinant by itself:
\begin{equation} \label{eq: def of det 2} 
    X_{T_i\sigma(T_i)}=\sum\limits_{\sigma\in S_k}(-1)^{sgn(\sigma)} x_{t_1\sigma(t_1)}\cdot\ldots\cdot x_{t_k\sigma(t_k)},
\end{equation}
for $T_i=\{t_1,\ldots,t_k\}$.
 
 We 
 call {\itshape the diagonal monomial} the monomial in $\det X((H\setminus e_1)\cup e_2;(H\setminus e_1)\cup S)$ corresponding to the choice of the identity permutations both in \eqref{eq: def of det} and
 in \eqref{eq: def of det 2} for every $1\le i\le f_{k-1}(H)$.

In order to prove that the matrix $X((H\setminus e_1)\cup e_2;(H\setminus e_1)\cup S)$ is invertible we will prove that the coefficient of the diagonal monomial in $\det X((H\setminus e_1)\cup e_2;(H\setminus e_1)\cup S)$ is nonzero. 

Since $|S\setminus e_2|=1$ and $|e_2\setminus S|=1$ the diagonal monomial has the form 
\begin{equation} \label{eq: form of the monomial}
    \prod\limits_{j=1}^{k(f_{k-1}(H)-1)+(k-1)} x_{i_ji_j}\ \cdot\ x_{pq},
    \end{equation}
where $p\ne q$, in fact $p=l_i$ and $q=y$.

Note that every nondiagonal entry of the matrix $X((H\setminus e_1)\cup e_2;(H\setminus e_1)\cup S)$ contains  
in each summand of its expansion a variable
$x_{pq}$ where the two indices differ, namely $p\ne q$.
Thus,
any permutation except for the identity contributes to \eqref{eq: def of det} only monomials with at least two nondiagonal entries.  
Thus,
monomials in $\det X((H\setminus e_1)\cup e_2;(H\setminus e_1)\cup S)$ of the form \eqref{eq: form of the monomial} can appear only in the summand in \eqref{eq: def of det} corresponding to the identity permutation. Any permutation on $[k]$ except for the identity contributes to \eqref{eq: def of det 2} at least two factors $x_{pq}$ with $p\ne q$. Thus, monomials in $\det X((H\setminus e_1)\cup e_2;(H\setminus e_1)\cup S)$ of the form \eqref{eq: form of the monomial} can appear only in the summand corresponding to the identity permutation in \eqref{eq: def of det} and only when picking the identity permutation in \eqref{eq: def of det 2} for every $1\le i\le f_{k-1}(H)$. However, this is the definition of the diagonal monomial. Thus, the only monomial of the form \eqref{eq: form of the monomial} is the diagonal monomial, so its coefficient is $1$.

 Since the matrix $X((H\setminus e_1)\cup e_2;(H\setminus e_1)\cup S)$ is invertible, the basis for the column space of the matrix $X((H\setminus e_1)\cup e_2;\binom{[n]}{k})$ chosen in the greedy way is not lexicographically greater than $(H\setminus e_1)\cup S$. 
Thus, 
$$\Delta^e((H\setminus e_1)\cup e_2)\le_{\text{lex}}(H\setminus e_1)\cup S.$$
However, since $S<_\text{lex}e_1$, we have that $$(H\setminus e_1)\cup S<_{\text{lex}} H=\Delta^e(H).$$
Taking these together we conclude that $\Delta^e((H\setminus e_1)\cup e_2)<_{\text{lex}}H$.
\end{proof}

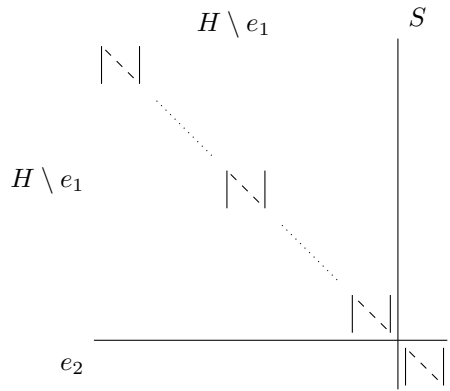
\begin{figure}[ht]
\centering
\begin{tikzpicture}
    \draw (-4,0) -- node [above left] {}(0,0);
    \draw (0,0) -- node [above right] {}(0,4);

    \draw (0.1,-0.1) -- node [above] {}(0.1,-0.6);
    \draw (0.6,-0.1) -- node [above] {}(0.6,-0.6);
    \draw[dashed] (0.15,-0.15) -- node [above right] {}(0.55,-0.55);

    \draw (-0.1,0.1) -- node [above] {}(-0.1,0.6);
    \draw (-0.6,0.1) -- node [above] {}(-0.6,0.6);
    \draw[dashed] (-0.15,0.15) -- node [above right] {}(-0.55,0.55);

    \draw (-3.9,3.9) -- node [above] {}(-3.9,3.4);
    \draw (-3.4,3.9) -- node [above] {}(-3.4,3.4);
    \draw[dashed] (-3.85,3.85) -- node [above right] {}(-3.45,3.45);

    \draw (-1.75,1.75) -- node [above] {}(-1.75,2.25);
    \draw (-2.25,1.75) -- node [above] {}(-2.25,2.25);
    \draw[dashed] (-2.2,2.2) -- node [above right] {}(-1.8,1.8);

    \draw[dotted] (-0.8,0.8) -- node [above right] {}(-1.55,1.55);
    \draw[dotted] (-2.45,2.45) -- node [above right] {}(-3.2,3.2);

    \draw (0,0) -- node [above right] {}(0.65,0);

    \draw (0,0) -- node [above right] {}(0,-0.65);

   \draw (-4,2.125) node [left]{$H\setminus e_1$};
   \draw (-1.55,4.19) node [left]{$H\setminus e_1$};
   \draw (-4,-0.35) node [left]{$e_2$};
   \draw (0.5,4.29) node [left]{$S$};
\end{tikzpicture}
\caption{Compound matrix $X((H\setminus e_1)\cup e_2;(H\setminus e_1)\cup S)$ and diagonal monomial.}
    \label{pic:cofactor matrix with S}
\end{figure}

\subsection{Proofs of main results} \label{sec: Proofs of main results}

In this section we prove Theorem \ref{thm: main thm 2 and 3} and
\ref{main theorem}.

\begin{proof}[Proof of Theorem \ref{thm: main thm 2 and 3}, $2$-uniform case.]

By Corollary \ref{Initial segment is a matroid} if the edges of $H$ form an initial lex-segment, then $H$ is matroidal.
For the opposite direction,
by Lemma \ref{lemma: reduction} we can assume that there is a subset $I\in \binom{[n]}{2}$ such that $1\in I$ and $I\notin H$. Denote by $\{1,i\}$ the lexicographically first such subset. Then by shiftedness of $H$ we conclude that $\{1,n\}\notin H$, i.e. $n$ is an isolated vertex. Since the edges of $H$ do not form an initial lex-segment, we conclude that $\{2,3\}\in H$. Thus, by Lemma \ref{lemma: partial star with an additional edge 2,3 is not a matroid} we get that $H$ is not matroidal.  
\end{proof}

Recall that $m_H$ denotes the maximum element of $H$ w.r.t. the lexicographic order, and the convention of writing in order the elements of a set.

\begin{proof}[Proof of Theorem \ref{thm: main thm 2 and 3}, $3$-uniform case.]
By Corollary \ref{Initial segment is a matroid} if the hyperedges of $H$ form an initial lex-segment, then $H$ is matroidal.
For the opposite direction, 
By Lemma \ref{lemma: reduction} we can assume that $\{2,3,4\}\in H$, and that there is a subset $I\in \binom{[n]}{3}$ such that $1\in I$ and $I\notin H$. Denote by 
$\{1,j,i\}$ the triple not in $H$ such that for every other triple $\{1,j',i'\}\notin H$, either $i<i'$, or $i=i'$ and $j<j'$. 
Equivalently, we can order all triples in as on Fig. \ref{pic:choice of skip} and start moving a vertical line from the right side to the left until we reach a situation where all triples from the left side of the line are hyperedges of $H$. Then, we take as $\{1,j,i\}$ the highest triple which is not a hyperedge of $H$ from the first column on the right side of the line.

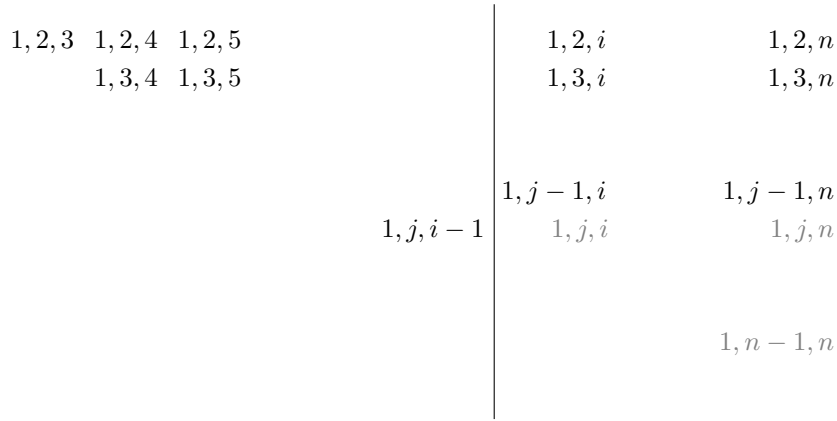
\begin{figure}[ht]
\centering
\begin{tikzpicture}
    \draw (5.4,-5) -- node [above left] {}(5.4,0.5);
    
   \draw (0,0) node [left]{$1,2,3$};
   \draw (1.1,0) node [left]{$1,2,4$};
   \draw (2.2,0) node [left]{$1,2,5$};

   \draw (7,0) node [left]{$1,2,i$};
   \draw (10,0) node [left]{$1,2,n$};

   \draw (1.1,-0.5) node [left]{$1,3,4$};
   \draw (2.2,-0.5) node [left]{$1,3,5$};
   \draw (7,-0.5) node [left]{$1,3,i$};
   \draw (10,-0.5) node [left]{$1,3,n$};

   \draw (7,-2) node [left]{$1,j-1,i$};
   \draw (10,-2) node [left]{$1,j-1,n$};
   \draw (5.4,-2.5) node [left]{$1,j,i-1$};
   \draw[gray] (7.03,-2.5) node [left]{$1,j,i$};
   \draw[gray] (10,-2.5) node [left]{$1,j,n$};

   \draw[gray] (10,-4) node [left]{$1,n-1,n$};
   
\end{tikzpicture}
\caption{Choice of $\{1ji\}$.}
    \label{pic:choice of skip}
\end{figure}
We now define a suitable permutation $\pi$ and prove that the pair $(H, \pi(H))$, where $\pi(H)=\{\pi(T):\ T\in H\}$, violates the exchange axiom among the elements of $M(H)$. In order to do that we have to provide an element $e_1\in H\setminus\pi(H)$ such that for every $e_2\in\pi(H)\setminus H$ we have that $(H\setminus e_1)\cup e_2\notin M(H)$.

We split the argument according to the following two cases.

\textbf{Case 1:} $\{2,3,n\}\in H$. Consider the transposition $\pi=(1,i-1)$ on $[n]$ written in one-row notation. 

By shiftedness of $H$ and the assumption $\{2,3,n\}\in H$, we have that $\{1,2,n\}\in H$. Note also that $\{1,2,n\}\notin \pi(H)$ since $\pi^{-1}(\{1,2,n\})=\{2,i-1,n\}>_p\{1,i-1,n\}\ge_p\{1,j,i\}\notin H$.
We use above that $i-1>2$, which holds since otherwise $\{1,2,3\}\notin H$ and thus $H$ would be empty. Take $e_1=\{1,2,n\}$.

As $e_2\in \pi(H)\setminus H$, by the definition of $\pi$, 
either $1\in e_2$ or $i-1\in e_2$. 
We exclude the former: 
if $1\in e_2$ 
then $\pi^{-1}(e_2)$ contains $i-1$, hence
 $\pi^{-1}(e_2)
>_p e_2\notin H$, so $\pi^{-1}(e_2)\notin H$, a contradiction.  
Thus, $i-1\in e_2$, and we split into two cases:

\textbf{1.} $e_2=\{i-1,x,y\}$ (not assumed to be ordered) for some $x<y$ and $y\ge  i$. 
Since $\{1,j,i\}\notin H$ also $\{1,i-1,y\}\notin H$
and we apply Lemma \ref{lemma: homology} (with $T=\{1,2,3,n\}$ and $\{i_2,i_3\}=\{i-1,y\}$).

\textbf{2.} $e_2=\{i-1,x,y\}$ (not assumed to be ordered) for some $x<y$ and $y<i$. 
Then $\max e_2=i-1<n=\max e_1$ and we
apply Lemma \ref{lemma: induced subgraph}.

\textbf{Case 2:} $\{2,3,n\}\notin H$.
Denote $m_H=:\{a,b,c\}$, in order.
We split the argument according to the following two subcases: either $a=2,b=3$ or not, and find a suitable permutation $\pi$ for each (we set $e_1=m_H$ for both cases).

\textbf{1.} $\{a,b\}\ne \{2,3\}$.

Define a permutation $\pi:[n]\to[n]$ written in the two-row notation as: 
    $$\pi=\left(\begin{smallmatrix}
1 & \ldots & a & \ldots & b-1 & b & \big| & \text{order}\\
1 & \ldots & a & \ldots & b-1 & n & \big| & \text{preserving}
\end{smallmatrix}\right),$$
i.e. the set $[b]$ is mapped as specified above and $\pi:[n]\setminus [b]\to[n]\setminus\pi([b])$ is the order-preserving bijection.

Take $e_1=m_H$. Note that $e_1\notin\pi(H)$. Indeed, 
if $c<n$ then $\pi^{-1}(m_H)=\{a,b+1,c+1\}>_p m_H$, thus $\pi^{-1}(m_H)\notin H$ by maximality of $m_H$, and if $c=n$ then $m_H=\{a,b,n\}\ge_p\{2,3,n\}\notin H$, a
contradiction to $H$ being shifted.

We deal separately with the cases $e_2<_{\text{lex}}m_H$ and $e_2>_{\text{lex}}m_H$ (since $e_2\in\pi(H)\setminus H$ and $m_H\in H$ we do not need to consider the case $e_2=m_H$):

$\bullet\  \bf{e_2<_{\text{lex}}m_H}$. By Lemma \ref{lemma: e_2<max} with $e_1=m_H$ we get that $\Delta((H\setminus m_H)\cup e_2)<_{\text{lex}}\Delta(H)$. Thus, $(H\setminus m_H)\cup e_2\notin M(H)$ and we are done.

$\bullet\  \bf{e_2>_{\text{lex}}m_H}$. Then, by the definition of $\pi$ and $m_H$, $e_2$ has the form $e_2=\{a,\pi(i),n\}$ for some $b<i\le c$. We argue by Lemma \ref{lemma: compound matrix with s} with $e_1=m_H$, and with $S=\{2,a,n\}$ if $a\ne2$ and $S=\{a=2,b-1,n\}$ otherwise (here $b-1\ne 2$ as we assumed $\{a,b\}\ne\{2,3\})$. Note that in either case $S<m_H$ and $S\notin H$ since $\{2,3,n\}\notin H$, so the lemma applies.

\textbf{2.} $\{a,b\}=\{2,3\}$.

Define a permutation $\pi:[n]\to[n]$ written in the two-row notation as: 
    $$\pi=\left(\begin{smallmatrix}
1 & 2 & 3 & \big| & \text{order}\\
1 & n-1 & n & \big| & \text{preserving}
\end{smallmatrix}\right),$$
i.e. the set $[3]$ is mapped as specified above and $\pi:[n]\setminus [3]\to[n]\setminus\pi([3])$ is the order-preserving bijection.

Set $e_1=m_H$. Note that $e_1\notin\pi(H)$ since 
$\{n-1,n\}\subset \pi^{-1}(e_1)$ but no element of $H$ contains $\{n-1,n\}$ as $\{1,j,i\}\le_p\{1,n-1,n\}$ and  $\{1,j,i\}\notin H$.

Again, we deal separately with the cases $e_2<_{\text{lex}}m_H$ and $e_2>_{\text{lex}}m_H$ (we do not need to consider the case $e_2=m_H$ as  $e_2\in\pi(H)\setminus H$):

$\bullet\  \bf{e_2<_{\text{lex}}m_H}$. By Lemma \ref{lemma: e_2<max} 
we get 
$\Delta((H\setminus m_H)\cup e_2)<_{\text{lex}}\Delta(H)$, thus $(H\setminus m_H)\cup e_2\notin M(H)$.

$\bullet\  \bf{e_2>_{\text{lex}}m_H}$. Then
$e_2$ has the form $e_2=\{\pi(i),n-1,n\}$ for some $b<i\le c$. We argue by Lemma~\ref{lemma: homology}, with $T=\{1,2,3,c\}$ and $i_2=n-1$, $i_3=n$. 

This completes the proof of the theorem.    
\end{proof}

We now turn to prove Theorem \ref{main theorem}. Recall that this theorem applies to  shifted $k$-uniform hypergraphs which do not contain a certain $k$-subset $s_H$ defined in terms of $m_H$. We define $s_H$ by \eqref{S for case 1}-\eqref{S for case 2} based on the following two cases:

\begin{enumerate}
\item[\bf{Case 1:}]  The maximum element of $m_H$ is not equal to $n$. Then $m_H$ has the form 
    \begin{equation} \label{eq: max in case 1}
    m_H=\{a,\ldots,b,c,c+1,\ldots,c+l,d\}
    \end{equation}
    with $d\ne n$ and $b< c-1$, where $l\in\mathbb Z_{\ge0}$. Note that possibly $c$ is the smallest element of $m_H$ in which case $m_H=\{c,c+1,c+2,\ldots,c+l,d\}$, namely $\{a,\ldots,b\}$ could be the empty set.

    If $c=1$, then $m_H=\{1,2,3,\ldots,l+1,d\}$. Then by the shiftedness of $H$ the hyperedges of $H$ form an initial lex-segment. Thus, we can assume $c\ne 1$. (Actually, by Lemma \ref{lemma: reduction} we can even assume that $1\notin m_H$, but we will not use it.)

Then define:
\begin{equation} \label{S for case 1}
s_H=\{a,\ldots,b,c-1,n-(l+1),n-l,\ldots,n-1\}.
\end{equation}

\item[\bf{Case 2:}] 
The maximum element of $m_H$ is equal to $n$. Then $m_H$ has the form 
\begin{equation} \label{eq: max in case 2}
m_H=\{a,\ldots,b,c,c+1,c+2,\ldots,c+l,n-m,n-m+1,\ldots,n-1,n\},
\end{equation}
where $b< c-1$, $l,m\in\mathbb Z_{\ge0}$. 
Note that $n-m$ might be the smallest element of $m_H$, and then $m_H=\{n-m,n-m+1,\ldots,n-1,n\}$, namely $\{a,\ldots,b,c,c+1,c+2,\ldots,c+l\}$ could be the empty set. However, if $n-m$ is the smallest element of $m_H$, $H$ is a complete hypergraph, and, in particular, its hyperedges form an initial lex-segment. Thus, from now on we will assume that $n-m$ is not the smallest element of $m_H$. In particular, that $c+l<n-m-1$. This implies that $c\in m_H$ (still $c$ might be the smallest one, and then $m_H=\{c,c+1,c+2,\ldots,c+l,n-m,n-m+1,\ldots,n-1,n\}$, namely $\{a,\ldots,b\}$ is an empty set).

If $c=1$, then $m_H=\{1,2,3,\ldots,l+1,n-m,n-m+1,\ldots,n-1,n\}$. Then by the shiftedness of $H$ the hyperedges of $H$ form an initial lex-segment. Thus, we can assume $c\ne 1$. (Again, by Lemma~\ref{lemma: reduction} we can even assume that $1\notin m_H$, but we will not use it). 

Then define:
\begin{equation} \label{S for case 2}
s_H=\{a,\ldots,b,c-1,n-m-l,n-m-l+1,\ldots,n-m-1,n-m,\ldots,n\}.
\end{equation}
\end{enumerate}

Note that in both cases $s_H$ is a $k$-subset of $[n]$ since $b\ne c-1$ and $c>1$.

\begin{proof}[Proof of Theorem \ref{main theorem}]
By Corollary \ref{Initial segment is a matroid} if the hyperedges of $H$ form an initial lex-segment, then $H$ is matroidal.
We now prove the opposite direction. 

We define a permutation $\pi:[n]\to[n]$
according to Case 1 and Case 2 separately:
\begin{enumerate}
\item[\bf{Case 1:}] 
In this case we define a permutation $\pi:[n]\to[n]$ written in the two-row notation as: 
    $$\pi=\left(\begin{smallmatrix}
a & \ldots & b & \big| & c & c+1 & c+2 & \ldots & c+l & d & \big| &\text{order}\\
a & \ldots & b & \big| & n-(l+1) & n-l & n-l+1 & \ldots & n-1 & n & \big| & \text{preserving}
\end{smallmatrix}\right),$$
i.e. the elements of $m_H$ are mapped as specified above and $\pi:[n]\setminus m_H\to[n]\setminus\pi(m_H)$ is the order-preserving bijection.

\begin{figure}[ht]
\centering
\begin{tikzpicture}

    \fill[black!30!white] (-5,0) -- (-3,0) -- (-3,0.3) -- (-5,0.3) -- (-5,0);

    \fill[black!30!white] (1.7,0) -- (2,0) -- (2,0.3) -- (1.7,0.3) -- (1.7,0);

    \fill[black!30!white] (-1,0) -- (1,0) -- (1,0.3) -- (-1,0.3) -- (-1,0);
    \fill[black!30!white] (2.4,0) -- (5,0) -- (5,-0.3) -- (2.4,-0.3) -- (2.4,0);

    \fill[black!30!white] (-5,0) -- (-3,0) -- (-3,-0.3) -- (-5,-0.3) -- (-5,0);

    \fill[black!30!white] (-1.3,0) -- (-1,0) -- (-1,0.3) -- (-1.3,0.3) -- (-1.3,0);

    \draw (-5,0) -- node [above left] {}(5,0);
    
    \draw (-5,0) -- node [above right] {}(-5,0.3);
    \draw (-3,0) -- node [above right] {}(-3,0.3);

    \draw (1,0) -- node [above right] {}(1,0.3);
    \draw (1.7,0) -- node [above right] {}(1.7,0.3);

    \draw (-5,0) -- node [above right] {}(-5,-0.3);
    \draw (-3,0) -- node [above right] {}(-3,-0.3);
    \draw (1.7,0) -- node [above right] {}(1.7,-0.3);
    \draw (5,0) -- node [above right] {}(5,-0.3);
    \draw (2,0) -- node [above right] {}(2,-0.3);
    \draw (2,0) -- node [above right] {}(2,0.3);
    \draw (5,0) -- node [above right] {}(5,0.3);

    \draw (1,0) -- node [above right] {}(1,-0.3);
    \draw (2.4,0) -- node [above right] {}(2.4,-0.3);
    \draw (2.4,0) -- node [above right] {}(2.4,0.3);

    \draw (-1.3,0) -- node [above left] {}(-1.3,0.3);
    \draw (-1.3,0) -- node [above left] {}(-1.3,-0.3);

   \draw (-5,0.3) node [above]{$a$};
   \draw (-3,0.3) node [above]{$b$};
   \draw (-0.6,0.3) node [above]{$c+1$};
   \draw (1,0.3) node [above]{$c+l$};
   \draw (1.85,0.3) node [above]{$d$};

   \draw (-4,-0.4) node [above]{$id$};
   \draw (5,0.3) node [above]{$n$};
   \draw (-1.15,0.32) node [above]{$c$};
\end{tikzpicture}
\caption{Image of $m_H$ under the permutation $\pi$, Case 1.}
    \label{pic: pi for case 1}
\end{figure}
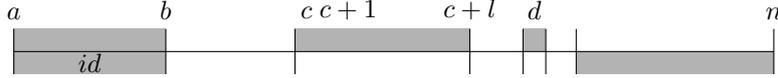

\item[\bf{Case 2:}] 
In this case we define a permutation $\pi:[n]\to[n]$ written in the two-row notation as: 
$$\pi=\left(\begin{smallmatrix}
a & \ldots  & b & \big| & c & c+1 & c+2 & \ldots & c+l & \big| & n-m & n-m+1 & \ldots & n-1 & n & \big| &\text{order}\\
a & \ldots & b & \big| & n-m-(l+1) & n-m-l & n-m-l+1 & \ldots & n-m-1 & \big| & n-m & n-m+1 & \ldots & n-1 & n & \big| & \text{preserving}
\end{smallmatrix}\right),$$
i.e. the elements of $m_H$ are mapped as specified above and $\pi:[n]\setminus m_H\to[n]\setminus\pi(m_H)$ is the order-preserving bijection.
\begin{figure}[ht]
\centering
\begin{tikzpicture}
    \fill[black!30!white] (-5,0) -- (-3,0) -- (-3,0.3) -- (-5,0.3) -- (-5,0);
    \fill[black!30!white] (-5,0) -- (-3,0) -- (-3,-0.3) -- (-5,-0.3) -- (-5,0);

    \fill[black!30!white] (3,0) -- (5,0) -- (5,0.3) -- (3,0.3) -- (3,0);
    \fill[black!30!white] (3,0) -- (5,0) -- (5,-0.3) -- (3,-0.3) -- (3,0);
    \fill[black!30!white] (-1,0) -- (1.5,0) -- (1.5,0.3) -- (-1,0.3) -- (-1,0);
    \fill[black!30!white] (0.5,0) -- (3,0) -- (3,-0.3) -- (0.5,-0.3) -- (0.5,0);

    \fill[black!30!white] (-1.3,0) -- (-1,0) -- (-1,0.3) -- (-1.3,0.3) -- (-1.3,0);

    \draw (-5,0) -- node [above left] {}(5,0);
    
    \draw (-5,0) -- node [above right] {}(-5,0.3);
    \draw (-3,0) -- node [above right] {}(-3,0.3);
    
    \draw (1.5,0) -- node [above right] {}(1.5,0.3);
    \draw (3,0) -- node [above right] {}(3,0.3);

    \draw (-5,0) -- node [above right] {}(-5,-0.3);
    \draw (-3,0) -- node [above right] {}(-3,-0.3);
    \draw (3,0) -- node [above right] {}(3,-0.3);
    \draw (5,0) -- node [above right] {}(5,-0.3);
    \draw (3,0) -- node [above right] {}(3,-0.3);
    \draw (3,0) -- node [above right] {}(3,0.3);
    \draw (5,0) -- node [above right] {}(5,0.3);

    \draw (0.5,0) -- node [above right] {}(0.5,-0.3);
    \draw (-1.3,0) -- node [above right] {}(-1.3,-0.3);
    \draw (-1.3,0) -- node [above right] {}(-1.3,0.3);

   \draw (-5,0.3) node [above]{$a$};
   \draw (-3,0.3) node [above]{$b$};
   \draw (-0.6,0.3) node [above]{$c+1$};
   \draw (1.5,0.3) node [above]{$c+l$};
   \draw (3,0.3) node [above]{$n-m$};

   \draw (-4,-0.4) node [above]{$id$};
   \draw (4,-0.4) node [above]{$id$};
   \draw (5,0.3) node [above]{$n$};
   \draw (-1.15,0.32) node [above]{$c$};

\end{tikzpicture}
\caption{Image of $m_H$ under the permutation $\pi$, Case 2.}
    \label{pic: pi in case 2}
\end{figure}
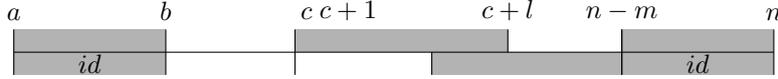

\end{enumerate}

In Figures \ref{pic: pi for case 1} and \ref{pic: pi in case 2} we pictorially explain the definition of $\pi$ according to Cases 1 and 2 respectively. The element $m_H\in H$ is drawn in grey above the horizontal line and its image $\pi(m_H)$ is drawn below the horizontal line.

\textbf{Claim 1.} Every $x\in\{1,\ldots,c-1\}$ is a fixed point of $\pi$, i.e. $\pi(x)=x$.

\begin{proof}[Proof of Claim 1]
In Case 1  $n-(l+1)> c$ since $d<n$ and in Case 2 $n-m-(l+1)> c$ by assumption that $c+l<n-m$. Thus, in either case, $\pi$ is the identity on $[c-1]\cap m_H$ and order-preserving on $[c-1]\setminus m_H$, hence the restriction of $\pi$ to $[c-1]$ is the identity map.
\end{proof}

We now prove that the pair $(H, \pi(H))$, where $\pi(H)=\{\pi(T):\ T\in H\}$, violates the exchange axiom. 
Again, 
we will
provide an element $e_1\in H\setminus\pi(H)$ such that for every $e_2\in\pi(H)\setminus H$ we have that $(H\setminus e_1)\cup e_2\notin M(H)$. Denote $m_H=:\{a_1,a_2,\ldots,a_k\}$.

\textbf{Claim 2.} In Case 1 every $e\in\pi(H)$ such that $e\ge_{\text{lex}}m_H$ has the following form:

\begin{equation} \label{eq: e2 for case 1}
e=\{a,\ldots,b,n-(l+1),n-l,\ldots,n-1\}\cup\{\pi(i)\},\ \text{where}\ c+l<i\le d.
\end{equation}
In Case 2 every $e\in\pi(H)$ such that $e\ge_{\text{lex}}m_H$ has the following form:
\begin{equation} \label{eq: e2 for case 2}
e=\{a,\ldots,b,n-m-(l+1),n-m-l,\ldots,n-m-1,n-m,\ldots,n\}.
\end{equation}

\begin{proof}[Proof of Claim 2]
Suppose $e\in\pi(H)$ is such that $e\ge_{\text{lex}}m_H$. Denote by $f=\{f_1,\ldots,f_k\}:=\pi^{-1}(e)$ the inverse image of $e$ under $\pi$. 
Since $f,m_H\in H$, and $m_H$ is the lex-maximum of $H$,  
either $ f=m_H$ or there exists a minimal $i\in[k]$ such that $f_i<a_i$. We now split our argument according to the cases in the definition of $\pi$.

Case 1: We prove that $f$ has the form $\{a,\ldots,b,c,c+1,\ldots,c+l\}\cup\{i\}$, where $c+l<i\le d$. 

First, we show that $ f $ starts with $a,\ldots,b,c$.
Indeed, suppose not, for a contradiction. Denote $t=|\{a,\ldots,b,c\}|$. 
Then as $f<_{\text{lex}} m_H$, 
the set of least $t$ elements of $f$ is lexicographically smaller than the set of least $t$ elements of $m_H$.
But 
then
$f$ contains an element of the set $ [c-1]\setminus m_H$, denote it $x$. By Claim 1, $\pi(x)=x$, hence $x\in e$. Thus, $e<_{\text{lex}}m_H$, a contradiction.
 Thus,
$$ f=\{a,\ldots,b,c,\ldots\}.$$
Since $f\le m_H$ and $m_H$ has the form \eqref{eq: max in case 1}, we get that 
$$ f=\{a,\ldots,b,c,c+1,c+2,\ldots,c+l,\ldots\}.$$

As $f$ has one more element than already specified we conclude that
$$ f=\{a,\ldots,b,c,c+1,\ldots,c+l\}\cup\{ i\},\ \text{where}\ c+l<i\le d.$$
Thus, by the definition of $\pi$ and by Claim 1, 
$$e=\{a,\ldots,b,n-(l+1),n-l,\ldots,n-1\}\cup\{\pi(i)\},\ \text{where}\ c+l<i\le d.$$

Case 2: Similarly to the previous case we prove that
$$ f=\{a,\ldots,b,c,c+1,c+2,\ldots,c+l,n-m,\ldots,n\}.$$

We have that $f$ starts with $a,\ldots,b,c$ by the same argument as the one used for Case 1. Thus, $f=\{a,\ldots,b,c,\ldots\}$.
Since $f\le m_H$ and $m_H$ has the form \eqref{eq: max in case 2}, we get that 
$f=\{a,\ldots,c,c+1,c+2,\ldots,c+l,\ldots\}$.

Note that $f$ has $m+1$ more elements. Suppose for contradiction that the set of greatest $m+1$ elements of $f$ is not the set $\{n-m,\ldots n\}$. Then, since 
$\pi(\{n-m,\ldots,n\})=\{n-m,\ldots,n\}$, we get that $m_H>\pi(f)=e$, contrary to our assumption $e\ge m_H$. Thus, 
$$ f=\{a,\ldots,b,c,c+1,\ldots,c+l,n-m,\ldots,n\}.$$

By the definition of $\pi$ and by Claim 1 
$$e=\{a,\ldots,b,n-m-(l+1),n-m-l,n-m-l+1,\ldots,n-m-1,n-m,\ldots,n\}.$$
\end{proof}

Note that since $\pi([c-1])=[c-1]$ (Claim 1) and $\pi(c)>c$, then $\pi(m_H)>m_H$ in both cases. Hence, $\pi(m_H)\notin H$. Thus, $H\ne \pi (H)$.

We now prove that by the definition of $\pi$ and Claim 2, $m_H$ actually belongs to the difference $H\setminus\pi(H)$. 

Indeed, by Claim 2 every $e\in\pi(H)$ such that $e\ge m_H$ has the form \eqref{eq: e2 for case 1} in Case 1 and \eqref{eq: e2 for case 2} in Case 2. In Case 1, by \eqref{eq: e2 for case 1} and assuming (for a contradiction) $e=m_H$, we get $c=\min(n-(l+1),\pi(i))$.  
If $c=n-(l+1)$ then $c+l=n-1$ so $\pi(i)=n\in e=m_H$, a contradiction; else $c=\pi(i)$ hence $c+1=n-(l+1),\ldots,c+l=n-2$ thus $d=n-1$, but then by the definition of $\pi$ we have $\pi(n)=c$. Combined, we get $\pi(n)=c=\pi(i)$, so $n=i\le d$. However, $d<n$, a contradiction. In Case 2 $\pi^{-1}(c)>c$, thus from \eqref{eq: e2 for case 2} we conclude that $e\ne m_H$. Hence $m_H\notin \pi(H)$.

Take $e_1=m_H$. Let $e_2 \in \pi(H) \setminus H$.

We deal separately with $e_2<_{\text{lex}}m_H$ and $e_2>_{\text{lex}}m_H$ (again, since $e_2\in\pi(H)\setminus H$ and $m_H\in H$ we do not need to consider the case $e_2=m_H$):

$\bullet\  \bf{e_2<_{\text{lex}}m_H}$. By Lemma \ref{lemma: e_2<max} 
we get $\Delta((H\setminus m_H)\cup e_2)<_{\text{lex}}\Delta(H)$, thus, $(H\setminus m_H)\cup e_2\notin M(H)$.

$\bullet\  \bf{e_2>_{\text{lex}}m_H}$.

 First, note that $|s_H\setminus e_2|=1$ and $|e_2\setminus s_H|=1$. This is a straightforward comparison of the equations \eqref{eq: e2 for case 1} with \eqref{S for case 1} and \eqref{eq: e2 for case 2} with \eqref{S for case 2}. Also, $s_H<_{\text{lex}}m_H$ by construction and $s_H\notin H$ by assumption.

 By Lemma \ref{lemma: compound matrix with s} with $e_1=m_H$ and $S=s_H$ we get $\Delta^e((H\setminus m_H)\cup e_2)<_{\text{lex}}H$. 
 
 Thus, the pair $(H,\pi(H))$ violates the exchange axiom among the elements of $M(H)$.
\end{proof}

\begin{example} \label{example: of not matroidal graph}
Consider the $2$-uniform hypergraph $H=\{12,13,23\}$ on $[5]$. Note that the (hyper)edges of $H$ do not form an initial lex-segment. Thus, by Theorem \ref{main theorem} $H$ is not matroidal. In this case one can verify this directly. Since algebraic shifting preserves $f$-vector and $\beta$-vector (see Theorem \ref{Basic properties of algebraic shifting}(3,4)) and the only shifted $2$-uniform hypergraph on $[5]$ with $3$ (hyper)edges and an one cycle is $H$, we conclude that $M(H)$ consists of all $2$-uniform hypergraph on $[5]$ with  
$3$ edges forming a triangle boundary.
However, the pair of elements
$\{12,13,23\},\ \{14,15,45\}\in M(H)$ does not satisfy the exchange axiom.
\end{example}

\begin{remark}\label{remark:Zemel} 
    Let us now restrict our attention to a class of shifted $k$-uniform hypergraphs which are given by picking a single $k$-subset and closing down w.r.t. the partial order $<_p$. Shifted pure simplicial complexes of this type were called {\itshape shifted matroids} in \cite[Section 4]{Ardila}, and {\itshape principal order ideals} in \cite[Section 5.4]{Klivans thesis}. For a shifted hypergraph which belongs to this class and which is not an initial lex-segment, we have $s_H\notin H$. It follows that Theorem \ref{main theorem} covers all shifted hypergraphs from this class. 
\end{remark}

\section{Matroids arising from symmetric shifting} \label{sec: Matroids arising from symmetric shifting}
Let $H$ be a shifted $k$-uniform hypergraph on $[n]$. 
Denote by
$$M^s(H)=\{G\subseteq \binom{[n]}k|\ \Delta^s(G)=H\}$$
the set of all $k$-uniform hypergraphs whose symmetric shifting is $H$.

\begin{definition}
Let $H$ be a shifted $k$-uniform hypergraph on $[n]$. $H$ is {\itshape s-matroidal} if $M^s(H)$ is the set of bases of a matroid on $\binom{[n]}{k}$.   
\end{definition}
\begin{proposition} \label{thm: symmetric case for graphs}
  A simple graph $H\subseteq\binom{[n]}{2}$ is s-matroidal if and only if the edges of $H$ form an initial lex-segment.  
\end{proposition}

The fact that the inverse image of an initial lex-segment w.r.t. the lexicographic order under exterior shifting is a set of bases of a matroid was proven via providing a concrete representation of this matroid. This representation was given by a compound matrix. The main goal of this section is finding an analogue of this result for the symmetric shifting of graphs. 

Let $x_1,x_2,\ldots,x_n$ be a set of $n$ variables, and $y_1,y_2,\ldots,y_n$ be their generic linear combinations  over the field $\mathbb F$. 
Let $B$ be some subset of all squarefree monomials in $x_i$'s of degree $2$, and $H$ some subset of squarefree monomials in $y_i$'s of degree $2$. 
Denote by $Y(B,H)$ the $n+|B|$ by $n+|H|$ matrix whose rows are labeled by the set $\{x_i^2|\ i\in[n]\}\cup B$, and columns are labeled by the set $\{y_1y_i|\ i\in[n]\}\cup S^{-1}_{un}(H)$, where $S^{-1}_{un}$ is the inverse of the unsquaring operation $S_{un}$, see Section \ref{subsection: algebraic shifting}; indeed the image of $S^{-1}_{un}$ involve only monomials in $y_2,\ldots,y_n$. Here, the rows and columns of $X(B,H)$ are ordered in such a way that $Y(B,H)$ is a block matrix whose top left block is labeled by $\{x_i^2|\ i\in[n]\}$ and $\{y_1y_i|\ i\in[n]\}$, and the bottom right block by $B$ and $S^{-1}_{un}(H)$, see Fig. \ref{pic:cofactor matrix in symm case}. The rows and columns in each block are labeled by the lexicographic order on monomials.

\begin{figure}[ht]
\centering
\begin{tikzpicture}
    \draw (-2,0) -- node [above left] {}(2,0);
    \draw (0,-2) -- node [above right] {}(0,2);
   
   \draw (-2.2,1) node [left]{$x_i^2$};
   \draw (-1,2.2) node [above]{$y_1y_i$};
   \draw (1,2.2) node [above]{$S^{-1}_{un}(H)$};
   \draw (-2.2,-1) node [left]{$B$};
\end{tikzpicture}
\caption{Matrix $Y(B,H)$.}
    \label{pic:cofactor matrix in symm case}
\end{figure}
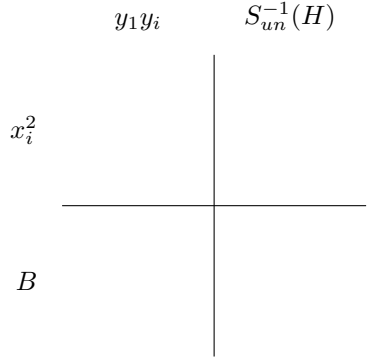

Any monomial in $y_i$'s can be decomposed uniquely in terms of monomials in $x_i$'s: $y_S=\sum\limits\alpha_{ST}x_T$, where $S$ and $T$ are multisets. We are now ready to define the entries of the matrix $Y(B,H)$: at the entry corresponding to the row $x_S$ and the column $y_T$ we put $\alpha_{ST}$.

For a set of degree $2$ squarefree monomials $H$ define a matroid $N'(H)$ on $\binom{[n]}{2}$ as follows: $B\subseteq\binom{[n]}{2}$ such that $|B|=|H|$ is a basis in $N'(H)$ if the rows of $Y(B,H)$ are linearly independent. 

For graphs $B,H$ we will write $Y(B,H)$ and $N'(H)$ identifying the graphs with the monomials spanned by their edges.

\begin{proposition} \label{cor: initial segment w.r.t. sym is a matroid}
 Let $H\subseteq\binom{[n]}2$ be a simple graph whose edges form an initial lex-segment. Then $H$ is s-matroidal.
\end{proposition}
\begin{proof}
    In order to prove that $M^s(H)$ is a matroid, we will prove that $M^s(H)=N'(H)$, where the latter is a matroid. 

    We first prove that $M^s(H)\subseteq N'(H)$. Consider a graph $G\in M^s(H)$, i.e. $G$ such that $\Delta^s(G)=H$. Since $\Delta^s(G)=H$ and by Theorem \ref{Basic properties of algebraic shifting}(3) algebraic shifting preserves $f$-vector, we know that $S^{-1}_{un}(G)\subseteq\text{GIN}(G)$ is the set of first $|H|$ monomials in the lexicographic order starting from $x_2^2$. By the same reason $\text{GIN}(H)$ contains $n$ monomials of degree $1$. However, there are only $n$ of them in total, namely, $y_1,\ldots,y_n$, so $\text{GIN}(H)$ contains $\{y_i|\ i\in[n]\}$. By the redundancy property of $\text{GIN}(H)$, 
    the other monomials that
    $\text{GIN}(H)$ contains are of the form $y_1y_i$ for every $i\in[n]$.  Thus, the columns of the matrix $Y(G,H)$ are linearly independent. 
    Hence so are the rows, and we conclude that $G\in N'(H)$.

    We now prove that $N'(H)\subseteq M^s(H)$. Consider a graph $G\in N'(H)$, i.e. $G$ such that the rows of the matrix $Y(G,H)$
    are linearly independent. If the rows of the matrix $Y(G,H)$ are linearly independent, then so are its columns. This means that a basis for the column space of $Y\left(G,\binom{[n]}{2}\right)$ chosen in the greedy way, is lexicographically not greater than $\{y_1y_i|\ i\in[n]\}\cup S^{-1}_{un}(H)$. On the other hand, the latter is an initial lex-segment, hence it cannot be smaller, so they are equal. Applying the operation $S_{un}$ to the monomials labeling these columns, we conclude that $\Delta^s(G)=H$.
\end{proof}

\begin{proof}[Proof of Proposition \ref{thm: symmetric case for graphs}]
If the hyperedges of $H$ form an initial lex-segment, then $H$ is matroidal by Proposition \ref{cor: initial segment w.r.t. sym is a matroid}. Conversely, note that in the proof of Theorem \ref{thm: main thm 2 and 3} for $k=2$ for the "only if" direction, we have used only properties of algebraic shiftings which are satisfied for both exterior and symmetric shifting, namely Lemmas \ref{lemma: partial star with an additional edge 2,3 is not a matroid} and \ref{lemma: reduction}. Thus, that proof is valid in the symmetric case as well.
\end{proof}

\section{Lex-matroidal hypergraphs} \label{sec: Some partial results}

For any $S\in \binom{[n]}k$ denote by
$$H(S)=\{T\in \binom{[n]}k|\ T\in H\text{ and }T\le_{\text{lex}}S\}$$
the sub-hypergraph of $H$ consisting of all hyperedges in $H$ that are lexicographically not greater than $S$.
\begin{definition}
Let $H$ be a shifted $k$-uniform hypergraph on $[n]$. $H$ is {\itshape lex-matroidal} if $H(S)$ is matroidal for every $S$.   
\end{definition}

We prove the analog of Conjecture~\ref{conj:any_k} for this stronger notion:

\begin{proposition} \label{thm: lex-matroidal}
    A shifted $k$-uniform hypergraph $H\subseteq\binom{[n]}{k}$ is lex-matroidal iff the hyperedges of $H$ form an initial lex-segment.
\end{proposition}
\begin{proof}
 By Corollary \ref{Initial segment is a matroid} if the hyperedges of $H$ form an initial lex-segment, then $H$ is matroidal. For every hyperedge $S$ of $H$ the set $H(S)$ is an initial lex-segment, and thus it is matroidal.

We now prove that if the hyperedges of $H$ do not form an initial lex-segment, then $H$ is not lex-matroidal. 

By Lemma \ref{lemma: reduction} we can assume that there is a set $\{1,i_2,\ldots,i_k\}$, where $i_2>2$, which is not the hyperedge of $H$ and that $\{2,3,\ldots,k,k+1\}$ is a hyperedge of $H$. Define a permutation $\pi:[n]\to[n]$ written in the two-row notation as:
$$\pi=\left(\begin{matrix}
3 & 4 & \ldots & k & k+1 & \big| &\text{order}\\
i_2 & i_3 & \ldots & i_{k-1} & i_{k} & \big| & \text{preserving}
\end{matrix}\right),$$
i.e. the elements of $\{3,4,\ldots,k,k+1\}$ are mapped as specified above and $\pi:[n]\setminus \{3,4,\ldots,k,k+1\}\to[n]\setminus\pi(\{3,4,\ldots,k,k+1\})$ is the order-preserving bijection.

Since $i_2>2$ we note that $\pi(1)=1$ and $\pi(2)=2$.

Note also that by shiftedness of $H$ and the fact that $\{2,3,\ldots,k,k+1\}$ is a hyperedge of $H$ we have that all facets of the boundary complex of the simplex $\{1,2,3,\ldots,k,k+1\}$ are hyperedges of $H$. Then $\pi(H)$ contains as hyperedges all facets of the boundary complex of the simplex $\{1,2,i_2,\ldots,i_{k}\}$.

Take $S=\{2,3,\ldots,k,k+1\}$.
By Theorem \ref{Basic properties of algebraic shifting}(2,6), $\Delta(\pi(H(S)))=\Delta(H(S))=H(S)$. $I':=\{1,i_2,i_3,\ldots,i_k\}\notin H(S)$, however $I'\in \pi(H(S))$ since $\pi(\{1,3,4,\ldots,k,k+1\})=I'$. Thus, $H(S)\ne \pi(H(S))$. We claim that the pair of $k$-hypergraphs $H(S),\pi(H(S))\in M(H(S))$ violates the exchange axiom among the elements of $M(H(S))$. 

Since $\pi(1)=1$ the only hyperedge of $\pi(H(S))$, which does not contain $1$, is $\pi(S)=\{2,i_2,i_3,\ldots,i_{k-1},i_k\}$. 
 By shiftedness of $H(S)$,
$\{2,i_2,i_3,\ldots,i_{k-1},i_k\}\ne \{2,3,\ldots,k,k+1\}$. Thus, $e_1=\{2,3,\ldots,k,k+1\}\in H(S)\setminus\pi(H(S))$. 
The deletion of the hyperedge $e_1$
form $H(S)$ decreases $\beta_{k-1}$, i.e. $\beta_{k-1}(H(S)\setminus e_1)<\beta_{k-1}(H(S))$. There is no hyperedge of $\pi(H(S))$ except possibly $\{2,i_2,\ldots,i_k\}$ whose addition to $H(S)\setminus e_1$ increases $\beta_{k-1}$, because all these hyperedges contain $1$. For $e_2=\{2,i_2,\ldots,i_k\}$ by Lemma~\ref{lemma: homology} $\Delta((H(S)\setminus e_1)\cup e_2)\ne\Delta(H(S))$, hence $(H(S)\setminus e_1)\cup e_2\notin M(H(S))$. Thus, the pair $H(S),\pi(H(S))$ violates the exchange axiom among the elements of $M(H(S))$, hence $H(S)$ is not matroidal. Thus, $H$ is not lex-matroidal.
\end{proof}

\begin{corollary}
    Let $H$ be a shifted $k$-uniform hypergraph which is the union of an initial lex-segment and another set $T$, such that $H$ is not an initial lex-segment. Then $H$ is not matroidal.
\end{corollary}
\begin{proof}
    Note that $H(S)$ for every $\{S\in H|\ S<_{\text{lex}}T\}$ is an initial lex-segment, and $H(S)=H(T)$ for every $S\ge_{\text{lex}}T$. Hence $H$ is not matroidal 
    iff $H(T)$ is not matroidal 
    iff $H$ is not lex-matroidal. Then the assertion follows from Proposition \ref{thm: lex-matroidal}.
\end{proof}

\section{Concluding remarks} \label{sec: Concluding remarks}
We state here several directions for further research related to our results.

We characterized s-matroidal simple graphs in Proposition \ref{thm: symmetric case for graphs}. The next natural question is to study s-matroidal hypergraphs.
\begin{problem}
    Characterize which shifted $k$-uniform hypergraphs are s-matroidal.
\end{problem}

Note that algebraic shifting can be defined not only w.r.t. the lexicographic order, but also w.r.t. any linear extension of the partial order $<_p$, see \cite[Section 2.1, Remark]{K Alg Sh}. Thus, one can extend the definition of matroidal hypergraph w.r.t. any linear extension of $<_p$ and try to characterize them.

\begin{remark}
    If $H\subseteq\binom{[n]}{k}$ is a shifted $k$-uniform hypergraph whose hyperedges form an initial segment w.r.t. some linear extension of the partial order $<_p$, then $H$ is matroidal in this extended sense if we shift w.r.t. the same linear extension. In order to see it, one can simulate the proof of Corollary \ref{Initial segment is a matroid}.
\end{remark}

The other direction is not clear.
\begin{problem}
    Characterize combinatorialy  which shifted $k$-uniform hypergraphs are matroidal when we shift w.r.t. an arbitrary linear extension of the partial order $<_p$.
\end{problem}
Any initial segment $H$ w.r.t. the lexicographic order gives rise to a matroid $M(H)$ which we can describe in two different ways. The first one is via exterior shifting: $M(H)=\{G\subseteq \binom{[n]}k|\ \Delta^e(G)=H\}$. The second one is via a concrete representation given by the matroid $M'(H)$, as described in Section \ref{sec: Matroids arising from exterior shifting}. 
For the Crapo-Rota simplicial matroid $M(H)$, where $H=\{S\subset \binom{[n]}{k}|\ 1\in S\}$, we have a third description which is a combinatorial characterization. Namely, the independent sets of $M(H)$ are the $(k-1)$-acyclic collections of $k$-subsets of $[n]$. 
 However, for other initial segments $H$ such a combinatorial characterization of the matroid $M(H)$ is not known. For area-rigidity, and more generally volume-rigidity, this was addressed in \cite[Problem 5.3]{BNP}. 
\begin{problem}
    Find a combinatorial characterization of the matroid $M(H)$ for every initial lex-segment~$H$.
\end{problem}
One can also hope to find a rigidity-style description for any matroid given by an initial lex-segment, like it exists, for example, for $H=\{\{i,j\}|\ \{i,j\}\le_{\text{lex}}\{k,n\}\}$ via hyperconnectivity and for $H=\{\{1,i,j\}|\ \{1,i,j\}\le_{\text{lex}}\{1,3,n\}\}$ via area-rigidity.
 
A computational direction is whether there exists a deterministic polynomial-time algorithm that given a uniform hypergraph checks its membership in $M(H)$ for a fixed initial lex-segment hypergraph $H$. In general, the computation of algebraic shifting involves the computation of symbolic determinants; see~\cite[Problem 8]{K Alg Sh}.
Even the case of $2$-hyperconnectivity, namely $H=\{A:\ A \le_{\text{lex}}\{2,n\}\}$, is open~\cite{Bernstein}.

\section*{Acknowledgements}
We are grateful to Shaul Zemel for pointing us to the class of shifted matroids, addressed in Remark~\ref{remark:Zemel}, and to Yuval Peled for inspiring discussions.


\begin{thebibliography}{}
\addcontentsline{toc}{section}{References}
 \bibitem{Ardila}  Federico Ardila, The Catalan Matroid. Journal of Combinatorial Theory, Series A, vol 104 no.1 49-62, 2003.
\bibitem{BNT} Eric Babson, Isabella Novik, and Rekha Thomas. Reverse lexicographic and lexicographic shifting. J. Algebraic Combin., 23(2):107–123, 2006.
\bibitem{Bernstein} Daniel Irving Bernstein,
  Completion of tree metrics and rank-2 matrices.
  Linear Algebra and its Applications,
  vol 533,
  1-13,
  2017.  
 \bibitem{BK} Anders Bj\"orner and Gil Kalai, An extended Euler-Poincar\'e theorem, Acta Math., 161, 279-303, 1988.
 \bibitem{BNP} Denys Bulavka, Eran Nevo and Yuval Peled, Volume rigidity and algebraic shifting, with 
J. Combin. Theory Ser. B 170, 189–202, 2023.
\bibitem{CR1} Henry H. Crapo and Gian-Carlo Rota, On the foundations of combinatorial theory: Combinatorial geometries. Preliminary edition. The M.I.T. Press, Cambridge, Mass.-London, 1970.
 \bibitem{CR2} Henry H. Crapo and Gian-Carlo Rota, Simplicial geometries. Combinatorics. (Proc. Sympos. Pure Math., Vol. XIX, Univ. California, Los Angeles, Calif., 1968), pp. 71–75. Amer. Math. Soc., Providence, R.I., 1971.
 \bibitem{EKR} Paul Erd\H{o}s, Chao Ko, and Richard Rado, Intersection theorems for systems of finite sets. Quart. J. Math. Oxford Ser. (2), 12:313–320, 1961.
 \bibitem{K ch} Gil Kalai, A characterization of f-vectors of families of convex sets in Rd, Part 1: Necessity of Eckhoff’s conditions, Israel j.Math., 48, 175-195, 1984.
 \bibitem{K Alg Sh} Gil Kalai, Algebraic shifting. In Computational commutative algebra and combinatorics (Osaka, 1999), volume 33 of
Adv. Stud. Pure Math., pages 121–163. Math. Soc. Japan, Tokyo, 2002.
 \bibitem{K Hyper} Gil Kalai, Hyperconnectivity of graphs, Graphs and Combi., 1, 65-79, 1985.
 \bibitem{K Symm} Gil Kalai, Symmetric matroids, J. Comb. Th. B., 50, 54-64, 1990.
 \bibitem{Klivans thesis} Caroline J. Klivans, Combinatorial Properties of Shifted Complexes, Ph.D. Thesis, MIT, 2003.
 \bibitem{Lee} Carl W. Lee. Generalized stress and motions. In Polytopes: abstract, convex and computational (Scarborough, ON, 1993), volume 440 of NATO Adv. Sci. Inst. Ser. C Math. Phys. Sci., pages 249–271. Kluwer Acad. Publ., Dordrecht, 1994.
 \bibitem{Nevo 2005} Eran Nevo, Algebraic shifting and basic constructions on simplicial complexes. J. Algebraic Combin., 22(4):411–433, 2005.
\bibitem{Nevo phd} Eran Nevo, Algebraic Shifting and f-Vector theory, Ph.D. Thesis, arXiv:0709.3265, 2007.
 \bibitem{N 2004} Eran Nevo, Embeddability and stresses of graphs. Combinatorica. 27, 465–472, 2007.
 \bibitem{Oxley} James Oxley, Matroid theory, Second edition, Oxford University Press, Oxford, 2011.
 \bibitem{White} Neil White, Combinatorial geometries. Encyclopedia of Mathematics and its Applications 29.
Cambridge University Press, Cambridge, 1987.

\end{thebibliography}
\end{document}